\documentclass[12pt]{article}

\usepackage{float}
\usepackage{amsfonts}
\usepackage{amssymb}
\usepackage{amsmath}
\usepackage{amsthm}
\usepackage{enumerate}
\usepackage{hyperref}

\usepackage{etex}
\usepackage{enumitem}
\usepackage{amsgen}
\usepackage{amsopn}
\usepackage{verbatim}
\usepackage{xspace}
\usepackage{multicol}
\usepackage{eepic}
\usepackage{url}
\usepackage{upref}
\usepackage{pgf}
\usepackage{tikz}
\usetikzlibrary{patterns}
\usepackage{pgffor}
\usepackage{pgfcalendar}
\usepackage{pgfpages}
\usepackage{shuffle,yfonts}
\usepackage{mathtools}
\usepackage{graphics}
\DeclareFontFamily{U}{shuffle}{}
\DeclareFontShape{U}{shuffle}{m}{n}{ <-8>shuffle7 <8->shuffle10}{}

\theoremstyle{plain}
\newtheorem{thm}{Theorem}[section]
\newtheorem{mainthm}[thm]{Main Theorem}
\newtheorem{lem}[thm]{Lemma}
\newtheorem{conj}[thm]{Conjecture}

\newtheorem{cor}[thm]{Corollary}

\theoremstyle{definition}
\newtheorem{defn}[thm]{Definition}
\newtheorem{rem}[thm]{Remark}
\newtheorem{eg}[thm]{Example}
\newtheorem{prob}[thm]{Problem}
\newtheorem{probs}[thm]{Problems}

\newtheorem{mainprob}[thm]{Main Problems}

\newcommand{\nc}{\newcommand}
\nc{\gl}{{\lambda}}
\nc{\vep}{{\varepsilon}}
\nc{\gd}{{\delta}}
\nc{\N}{{\mathbb N}}
\nc{\calO}{{\mathcal O}}
\nc{\calS}{{\mathcal S}}
\nc{\calT}{{\mathcal T}}
\nc{\calD}{{\mathcal D}}
\nc{\calM}{{\mathcal M}}
\nc{\bfn}{{\mathbf n}}

\begin{document}

\title{\bf The Largest Circle Enclosing $n$ Lattice Points}
\author{
Jianqiang Zhao\thanks{Email: zhaoj@ihes.fr, ORCID 0000-0003-1407-4230.}\\ [2mm]
Department of Mathematics, The Bishop's School, La Jolla, CA 92037, USA}
\date{}
\maketitle

\noindent
{\bf Abstract.}
In this paper, we propose a class of elementary plane geometry problems closely related to the title of this paper. Here, a circle is the 1-dimensional curve bounding a disk. For any nonnegative integer, a circle is called $n$-enclosing if it contains exactly $n$ lattice points on the $xy$-plane in its interior. The main questions are when the largest $n$-enclosing circle exists and what the largest radius is. We study the small integer cases by hand and extend to all $n<1100$ with the aid of a computer. We find that frequently such a circle does not exist, e.g., when $n=5,6$. We then show a few general results on these circles including some regularities among their radii and an easy criterion to determine exactly when the largest $n$-enclosing circles exist. Further, from numerical evidence, we conjecture that the set of integers whose largest enclosing circles exist is infinite, and so is its complementary in the set of nonnegative integers. Throughout this paper we present more mysteries/problems/conjectures than answers/solutions/theorems. In particular, we list many conjectures and some unsolved problems including possible higher dimensional generalizations at the end of the last two sections.


\section{Introduction}
There are many naive-looking number theory problems, such as the twin prime conjecture and the Goldbach conjecture, whose solutions are extremely difficult and remain unsolved until this day. Many geometry problems have also fascinated human kind for millennia, e.g.,  the Pythagorean theorem, known to Babylonians well before Pythagoras and named ``Gou-Gu theorem'' in the East, and the much deeper Poincare conjecture on sphere packing, one of the Millennium Prize Problems.

In this paper, we propose to study a class of seemingly naive plane geometry problems concerning circles and lattice points on the $xy$-plane. Here, a circle means the 1-dimensional circumference of a disk and a lattice point is a point on the $xy$-plane with integer coordinates.

We first recall a famous result of H.\ Steinhaus first stated as Problem 24 on page 17 of his book \cite{Steinhaus1964} and then answered on page 84.
\begin{thm}\label{thm:Steinhaus}
For every positive integer $n$, there exists a circle of area $n$ enclosing exactly $n$ lattice points in its interior.
\end{thm}
In fact, the original problem did not have the condition on the area. Honsberger \cite[p.~118]{Honsberger1973} further strengthened it by restricting the area to be $n$. Several years earlier, Schinzel \cite{Schinzel1958} showed that for every positive integer $n$, there exists a circle on the plane having exactly $n$ lattice points on its circumference, which was given the name \emph{Schinzel circle}. Also see \cite[p.~118]{Honsberger1973}. For any circle of radius $r$ centered at the origin, it is well known that the number $N(r)$  of lattice points in the interior \emph{and on the boundary} is given by
\begin{equation}\label{equ:NumberOfPt}
     N(r)=1+4\lfloor r\rfloor+4\sum_{j=1}^{\lfloor r\rfloor} \lfloor \sqrt{r^2-j^2}\rfloor.
\end{equation}
See, e.g., \cite{Hilbert1999} or \cite[Eq. (8)]{FraserGotlieb1962}.
The sequence $\{N(k)\}_{k\ge 1}$ is the sequence [A000328] on the OEIS website \cite{Sloane2025}.
By considering the area one can see that the number $N(r)$ is close to $\pi r^2$ for large $r$. In particular, one may find the following estimate credited to Gauss in Hardy's book \cite[p.~67]{Hardy1999}:
\begin{equation}\label{equ:GaussBound}
    N(r)=\pi r^2+E(r), \quad \text{where}\quad |E(r)|\le 2\sqrt{2} \pi r.
\end{equation}
We will present a related result in Theorem \ref{thm:RadiusBound} with a short proof. Incidentally, the error bound for $E(r)$ can be further improved which turns out to be a significant problem having close connection with analytic number theory and therefore has attracted quite some work throughout the past century. See \cite{Chen1963,Cilleruello1993,FraserGotlieb1962,Hua1942,Huxley,Iwaniec1988,Keller1963}.

In another direction, we found the following as Problem \textbf{634} on p.\ 57 of \cite{exeter2017}:
\begin{quote}
What is the radius of the largest circle that you can draw on graph paper that encloses
\begin{alignat*}{4}
&\text{(a) no lattice points?} &&\text{(b) exactly one lattice point?}\\
&\text{(c) exactly two lattice points?}\qquad &&\text{(d) exactly three lattice points?}
\end{alignat*}
\end{quote}

To answer the above questions it is intuitively helpful to draw the pictures using some computer software so that the radii of the circles can be varied continuously. We have used \textsf{Geogebra} heavily in our experiments.

To generalize the above questions we need to introduce some terminology. We call a nonnegative integer $n$ \emph{maximally circlable}\footnote{The word ``circlable'' does not really exist in English. Perhaps a more precise version should be \emph{maximally enclosable by a circle}. But the acronym MEBAC is still too long compared to MC.} if  there exists a largest circle enclosing exactly $n$ lattice points in its interior. Otherwise $n$ is called \emph{non-maximally circlable} (non-MC). Since such circles might not be unique what we mean by ``largest'' is the largest radius of these circles. For brevity, we will often drop ``exactly'' and ``in its interior'' in the rest of the paper.

For convenience, we fix the \emph{key triangle} $\Delta$ as the triangle with vertices at $(1/2,0)$, $(1/2,1/2)$, and $(1,0)$, including both its interior and its boundary. If $n$ is MC then we let $\calM_n$ be a largest circle enclosing $n$ lattice points with its center $O\in \Delta$, called an \emph{MC-circle  of $n$}, and let $R_n$ be its radius, called the \emph{MC-radius of $n$}. It is not clear if such a largest circle is always unique even with this restriction of its center. However, by rigid motions of the lattice system (i.e., isometries keeping the grid invariant)\footnote{It is easy to see that all grid invariant isometries are given by translations by integer values, reflections about $x,y=$ integers or half integer or $y=\pm x+$ integers, and rotations by multiples of $90^{\circ}$.}, every MC-circle of $n$ can be transformed to such an $\calM_n$. We call a set of $n$ lattice points a \emph{MC-set of $n$} if an MC-circle of $n$ encloses them. A circle is called a \emph{lattice circle} if it has at least three lattice points on its circumference. From overwhelming evidence, we formulate the following conjectures.

\begin{conj}\label{conj:MaxCircleUnique}
For any MC integer $n$, both the configuration of the MC-set and the MC-circle of $n$ are unique up to rigid motions of the lattice system.
\end{conj}

We want to caution the reader that for a fixed integer $n$ there may be different $n$-enclosing lattice circles with different configurations of the $n$-lattice points even under rigid motions. But most of these circles are not the MC-circle of $n$. See, e.g., the two $89$-enclosing circles on page \pageref{fig:case89}.

\textbf{Key Notation and Terminology.} Throughout the paper, we adopt the convention that $r_n=R_n$ means $n$ is MC and $R_n$ is its MC-radius. On the other hand, the notation $r_n\prec R_n$ means $n$ is non-MC and $R_n$ is the \emph{least upper bound of the radii} of all the $n$-enclosing circle. This is a little mouthful and therefore we will use the acronym \emph{LUBOR} of $n$ in the rest of the paper. We also emphasize that \textbf{the notation $r_n$ and $R_n$ are always reserved for this purpose}. When we say ``enclose'' we always mean to contain inside but not on the boundary/circumference.

We will primarily focus on the following questions in this paper.

\begin{mainprob}  \label{mainprob}
Let $n$ be a nonnegative integer.
\begin{enumerate}[label=(\arabic*)]
  \item \label{mainprob:Case1} When is $n$ MC?  When is $n$ non-MC? Are there good criteria?

  \item  \label{mainprob:Case2}  How to determine MC-radius of $n$ if $n$ is MC?

  \item  \label{mainprob:Case3}  How to determine LUBOR of $n$ if $n$ is non-MC?

  \item  \label{mainprob:Case4}  Is there a largest $n$-enclosing \textbf{lattice} circle?
\end{enumerate}
\end{mainprob}

We will answer Main Problem \ref{mainprob}\ref{mainprob:Case1}--\ref{mainprob:Case4}  for all $n<1100$. In this range, we also find that the answer to Main Problem \ref{mainprob}\ref{mainprob:Case4} is affirmative for all $n\ne 6$ but cannot confirm this in general at present. However, we can show that the MC-circles of $n$ are all lattice circles in Theorem \ref{thm:MCLatticeCircle}.

For examples, we will prove in Theorem \ref{thm:R0-4} that $R_4=\sqrt{10}/2$ which means the largest circle enclosing $4$ lattice points exists and its radius is $\sqrt{10}/2$. It is quite obvious that four lattice points form a MC-set of $4$ if and only if they form a unit square, which confirms Conjecture \ref{conj:MaxCircleUnique}.

In Theorem \ref{thm:R5-6}, we will prove that $r_{5}\prec R_4$ and $r_{6}\prec R_4$, i.e., the largest circles enclosing $5$ or $6$ lattice points do not exist and their LUBOR are both $\sqrt{10}/2$. This phenomenon relating the LUBOR of a non-MC number to some MC-radius of a smaller MC number should continue to hold for other non-MC numbers. We confirmed this in Theorem \ref{thm:FindNonMaxRadius} providing a satisfactory answer to Main Problem \ref{mainprob}\ref{mainprob:Case3}.

Concerning the relations between different MC-radii and the LUBORs, we have the following main results.
\begin{mainthm}\label{mainthm:AllResults}
Let $\{R_n: n\ge 0\}$ be the positive sequence of MC-radii (for MC $n$) and LUBORs (for non-MC $n$).
Then we have:
\begin{enumerate}[label=(\arabic*)]
  \item \label{mainthm:Case1} (=Theorem \ref{thm:RnIncrease}) $\{R_n: n\ge 0\}$ is a non-decreasing sequence.

  \item \label{mainthm:Case3}  (=Theorem \ref{thm:FindNonMaxRadius}) Suppose $k<\ell$ are two consecutive MC numbers and $n$ is non-MC such that $k<n<\ell$. Then $R_n=R_k$, namely, $R_k$ is the LUBOR for all such non-MC $n$ immediately following $k$.

  \item \label{mainthm:Case4}  (=Corollary \ref{cor:NonMCradiiStrictDecrease}) Suppose $k<n$, $k$ is an MC number and all the numbers between $k$ and $n$ (exclusive) are non-MC. Then the radius of the largest $n$-enclosing \textbf{lattice} circle is \textbf{strictly} smaller than $R_k$ if and only if $n$ is non-MC.

\end{enumerate}
\end{mainthm}

We now briefly describe the content of this paper. We start with small integers in the next section and provide a few essential ideas to be used in the general case. In Section \ref{sec:general}, we state a few general results including those in the Main Theorem, which provide the foundation for our computer computation in  Section \ref{sec:computer}.

Section \ref{sec:SpecialClass} focuses on two special classes of lattice circles which likely will include infinitely many MC-circles. Most of our computational results are then presented in  Section \ref{sec:computer} which completely solves the Main Problems \ref{mainprob} for all $n<1100$. We then conclude the paper with a brief summary and some further problems for possible future research at the end of the last two sections.


\section{Small numbers: $n\le 18$}\label{sec:smallNumbers}
We can treat the cases with $n\le 4$ through pure geometric argument.

\begin{thm} \label{thm:R0-4}
The MC-radii $R_0=\sqrt{2}/2$, $R_1=1$, $R_2=\sqrt{5}/2$, $R_3=5\sqrt{2}/6$, and $R_4=\sqrt{10}/2$.
\end{thm}
\begin{proof}
We give detailed proofs of these cases because, although they are simple, they reflect almost most of the main ideas one can use to prove the general MC cases.

First, the circles in Figure~\ref{fig:R0-4} imply the lower bounds of  $R_n$ for $0\le n\le 4$. Now we need to show the opposite inequalities so that $R_n$ are equal to these values, respectively.

\begin{figure}[H]
\centering
  \includegraphics[scale=0.6]{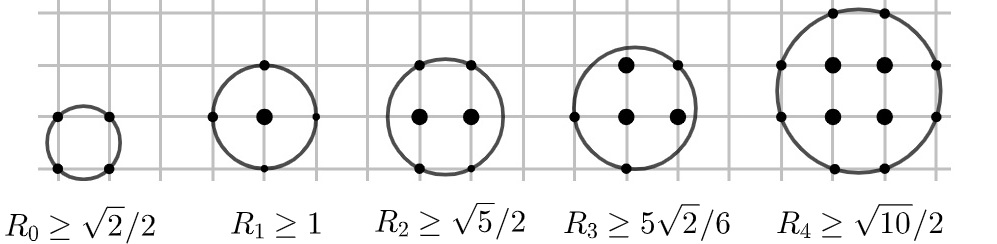}
\caption{The largest circles enclosing zero to four lattice points.}
\label{fig:R0-4}
\end{figure}

$n=0$. Let $O$ be the center of a circle whose radius $r>\sqrt{2}/2$. We now show that it must enclose at least one lattice point, which implies that $R_0\le \sqrt{2}/2$ and so $R_0=\sqrt{2}/2$. Indeed, as we explained in the introduction we may assume $O\in\Delta$ as shown in the picture
\raisebox{-3pt}{\begin{tikzpicture}[scale=0.5]
\draw (0,0)--(1,0)--(1,1)--(0,1)--(0,0);
\draw[dotted] (0,0)-- (1,1);
\draw[dotted] (1,0)--(0,1);
\filldraw [fill=black] (0.6,0.2) circle (1pt);
\end{tikzpicture}}.
Then the lower right corner point $(1,0)$ is at most $\sqrt{2}/2$ away from $O$ and therefore must lie in the interior of the circle $O$ of radius $r>\sqrt{2}/2$.

$n=1$. Let $O\in\Delta$ be the center of a circle whose radius $r>1$. We now show that it must enclose at least two lattice points, which implies that $R_1\le 1$ and so $R_1=1$. By the picture
\raisebox{-3pt}{\begin{tikzpicture}[scale=0.5]
\draw (0,0)--(1,0)--(1,1)--(0,1)--(0,0);
\draw[dotted] (0,0)-- (1,1);
\draw[dotted] (1,0)--(0,1);
\filldraw [fill=black] (0.6,0.2) circle (1pt);
\end{tikzpicture}},
the bottom two vertices $(0,0)$ and $(1,0)$ are at most $1$ unit away from $O$ and therefore must lie in the interior of the circle $O$ of radius $r>1$.

$n=2$. Let $O\in\Delta$ be the center of a circle whose radius $r>\sqrt{5}/2$. We now show that it must enclose at least three lattice points, which implies that $R_2\le \sqrt{5}/2$ and so $R_2=\sqrt{5}/2$. By the picture
\raisebox{-3pt}{\begin{tikzpicture}[scale=0.5]
\draw (0,0)--(1,0)-- (1,1)--(0,1)--(0,0);
\draw[dotted] (0,0)--(1,1);
\draw[dotted] (0.5,0)--(0.5,1);
\draw[dotted] (1,0)--(0,1);
\draw[dotted] (0,0.5)--(1,0.5);
\filldraw [fill=black] (0.6,0.2) circle (1pt);
\end{tikzpicture}}
the bottom two lattice vertices $(0,0)$ and $(1,0)$ already lie in the interior of the circle $O$ of radius $r>\sqrt{5}/2>1$.
Further, the upper right corner point $(1,1)$ of the square is at most $\sqrt{5}/2$ away from $O$ and therefore must lie in the interior of the circle $O$ of radius $r>\sqrt{5}/2$.

$n=3$. Let $O\in\Delta$ be the center of a circle whose radius $r>5\sqrt{2}/6$. We now show that it must enclose at least four lattice points, which implies that $R_3=5\sqrt{2}/6$. Continuing the argument in the case of $n=2$ we may assume the left picture in Figure~\ref{fig:R3} in which the three lattice points $A$, $B$ and $C$ are already in the interior of the circle $O$.

\begin{figure}[H]
\centering
\begin{tikzpicture}[scale=1.5]
\draw (0,0) node[below]{$A$}--(1,0) node[below]{$B$}--(1,1) node[above]{$C$}--(0,1) node[above]{$D$}--(0,0);
\draw[dotted] (0,0)--(1,1);
\draw[dotted]  (0.5,0) node[below]{$M$} -- (0.5,1);
\draw[dotted] (1,0)--(0,1);
\draw[dotted] (0,0.5)--(1,0.5);
\filldraw [fill=black] (0.6,0.2) circle (0.5pt);
\draw (0.6,0.2) node[right]{$O$};
\end{tikzpicture} \hskip2cm
\begin{tikzpicture}[scale=1.5]
\draw (0,0) node[below]{$B$}--(1,0) node[below]{$C$}--(1,1) node[above]{$D$}--(0,1) node[above]{$A$}--(0,0)-- (-1,0) node[below]{$E$}--(-1,1)--(0,1);
\draw[dotted] (0,0)--(1,1);
\draw[dotted] (1,0)--(0,1);
\draw[dotted] (0.5,0) --(0.5,1);
\draw (0.5,0.5) node[right]{$N$};
\draw[dotted] (0,0.5) node[left]{$M$}--(1,0.5);
\filldraw [fill=black] (0.2,0.4) circle (0.5pt);
\draw (0.2,0.4) node [right]{$O$};
\draw[dotted] (0,0.5)--(0.163,0.163) node[right]{$Q$};
\draw[dashed] (-1,0)--(1,1);
\filldraw [fill=black] (0.163,0.163) circle (0.5pt);
\draw[dashed] (-1,0)--(0.163,0.163);
\end{tikzpicture}
\caption{Circles with radius $r>5\sqrt{2}/6$ must enclose at least four lattice points.}
\label{fig:R3}
\end{figure}
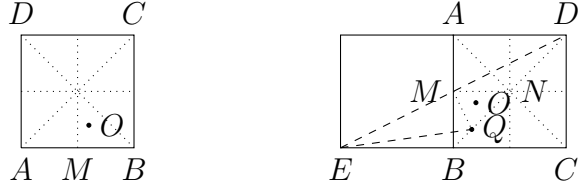

Let $N$ be the center of $ABCD$. To save space, we rotate the picture clockwise by 90${}^\circ$.
Now mark the point $Q$ on $BD$ so that $MQ\bot ED$ as shown by the right picture in Figure~\ref{fig:R3}. Then $|BQ|=\sqrt{2}/6$ and $MQ$ divides $\triangle BMN$ into two smaller triangles. If $O$ is in $\triangle QMN$ then
$$|OD|\le \max\{|MD|,|QD|\}=5\sqrt{2}/6.$$
If $O$ is in $\triangle QMB$ then
$$|OE|\le \max\{|ME|,|QE|, |BE|\}=|QE|=\sqrt{(1/6)^2+(7/6)^2}=5\sqrt{2}/6.$$
So the fourth point, either $D$ or $E$, lies in the interior of the circle $O$ of radius $r>5\sqrt{2}/6$.

$n=4$. This case turns out to be easy because of the apparent symmetry. Let $O\in\Delta$ be the center of a circle whose radius $r>\sqrt{10}/2$. Since $r>\sqrt{2}$ it is easy to see that all the four corners of the unit square $ABCD$ must lie in the interior of the circle.
Let $N$ be the center of $ABCD$ and rotate the square clockwise by 90${}^\circ$ to save space.

\begin{figure}[H]
\centering
\begin{tikzpicture}[scale=1.5]
\draw (0,0) node[below]{$A$}--(1,0) node[below]{$B$}--(1,1) node[above]{$C$}--(0,1) node[above]{$D$}--(0,0);
\draw[dotted] (0,0)--(1,1);
\draw[dotted]  (0.5,0) node[below]{$M$} -- (0.5,1);
\draw[dotted] (1,0)--(0,1);
\draw[dotted] (0,0.5)--(1,0.5);
\filldraw [fill=black] (0.6,0.2) circle (0.5pt);
\draw (0.6,0.2) node[right]{$O$};
\end{tikzpicture} \hskip2cm
\begin{tikzpicture}[scale=1.5]
\draw (0,0) node[below]{$B$}--(1,0) node[below]{$C$}--(1,1) node[above]{$D$}--(0,1) node[above]{$A$}--(0,0)-- (-1,0) node[below]{$E$}--(-1,1)--(0,1);
\draw[dotted] (0,0)--(1,1);
\draw[dotted] (1,0)--(0,1);
\draw[dotted] (0,0.5)--(1,0.5);
\draw[dashed] (0,0.5)--(-0.1,0.5)node[above]{$M$}--(-1,0.5) node[left]{$K$};
\draw[dotted] (0.5,0) node[below]{$L$}--(0.5,1);
\filldraw [fill=black] (0.2,0.4) circle (0.5pt);
\draw (0.2,0.4) node [right]{$O$};
\draw (0.5,0.5) node[above]{$N$};
\end{tikzpicture}
\end{figure}

Then
$$|OE|\le |NE|=\sqrt{10}/2$$
since $O$ is in the rectangle $EKNL$ with diagonal $NE$.
Hence five lattice points $A$, $B$, $C$, $D$ and $E$ are all in the interior of the circle $O$.
\end{proof}

Moving on to the next case, we will see that the largest circle enclosing exactly 5 lattice points does not exist. Namely, 5 is the first non-MC number. We now give detailed proof of this case which contains basically all the main ideas one may use to prove the general non-MC cases. We first prove the following two lemmas.

\begin{lem}\label{lem:n=5Step1}
Every circle of radius $r=\sqrt{10}/2$ must enclose exactly 4 lattice points or at least 7 lattice points. For any real number $r>\sqrt{10}/2$, every circle of radius $r$ must enclose at least 7 lattice points.
\end{lem}
\begin{proof}
Let $O\in\Delta$ be the center of a circle whose radius $r\ge\sqrt{10}/2$. If $O=(1/2,1/2)$ and $r=\sqrt{10}/2$ then clearly the circle encloses exactly four lattice points as shown in Figure~\ref{fig:R0-4}. If $O=(1/2,1/2)$ but $r>\sqrt{10}/2$ then clearly the circle encloses at least 12 lattice points.

Now we assume $O\ne (1/2,1/2)$ and $r\ge\sqrt{10}/2$. Then the circle encloses all the four corners of the unit square since $r>\sqrt{2}$. We now show it encloses at least three more lattice points.
We now add two adjacent unit squares as shown in the picture of Figure~\ref{fig:R67}.
\begin{figure}[H]
\centering
\begin{tikzpicture}[scale=1.5]
\draw (0,0) node[anchor=north east]{$B$}--(1,0) node[anchor=north west]{$C$}--(1,1) node[above]{$D$}--(0,1) node[above]{$A$}--(0,0);
\draw (1,0)--(1,-1) node[below]{$E$}--(0.5,-1) node[below]{$M$}--(0,-1) node[below]{$F$}--(0,0);
\draw[dashed] (0.5,0)node[anchor=north east]{$K$} --(0.5,-1) ;
\draw[dotted] (0.5,0)--(0.5,1);
\draw[dotted] (0,0)--(1,1);
\draw[dotted] (1,0)--(0,1);
\draw[dotted] (0,0.5)--(1,0.5);
\draw (1,0.6)node[right]{$L$};
\draw (1,0)--(2,0) node[below]{$G$} -- (2,0.5) node[right]{$Q$}-- (2,1)--(1,1);
\draw[dotted] (2,0.5)--(1,0.5);
\draw (0.5,0.5) node[above]{$N$};
\filldraw [fill=black] (0.6,0.2) circle (0.5pt);
\draw (0.6,0.2) node[right]{$O$};
\draw[dashed] (0,-1)--(1,0);
\end{tikzpicture}
\caption{Circles with radius $r\ge\sqrt{10}/2$ whose center is not at the center of any unit square must enclose at least seven lattice points.}
\label{fig:R67}
\end{figure}
Since $O\ne N=(1/2,1/2)$, we have
\begin{equation*}
 |OE|< |NE|=\sqrt{10}/2, \qquad
 |OG|< |NG|=\sqrt{10}/2.
\end{equation*}
since $O$ is in the rectangle $EMNL$ with diagonal $NE$ and in the rectangle $GKNQ$ with diagonal $NG$. Further, in the right triangle $\triangle NFC$
$$ |OF|< |NF|=\sqrt{10}/2.$$
Thus the three lattice points $G$, $E$ and $F$, together with $A, B, C,$ and $D$, are all in the interior of the circle $O$. This completes the proof of the lemma.
\end{proof}

To prove the next lemma, we need borrow some ideas from calculus involving the concepts of limit and continuity.

\begin{lem}\label{lem:n=5Step2}
There exists a small $\vep>0$ such that for any real number $r\in(\sqrt{10}/2-\vep,\sqrt{10}/2)$, there is a circle $O_5(r)$ of radius $r$ enclosing exactly five lattice points and a circle $O_6(r)$ of radius $r$ enclosing exactly six lattice points.
\end{lem}
\begin{proof}
The proof uses essentially the same idea as the one for Theorem \ref{thm:RnIncrease} in the next section. One can also just use that theorem instead of this ad hoc lemma to prove Theorem \ref{thm:R5-6}.
\end{proof}

By combining the above two lemmas we achieve the following result immediately.
\begin{thm}\label{thm:R5-6}
We have $r_n\prec R_n=\sqrt{10}/2$ for $n=5, 6$.
\end{thm}
\begin{proof}
Lemma \ref{lem:n=5Step1} implies that $R_n\le R_4$ for $n=5, 6$, while Lemma \ref{lem:n=5Step2} (or Theorem \ref{thm:RnIncrease} more precisely) shows that $R_n\ge R_4$, whence the claim in the theorem.
\end{proof}

We can mimic the ideas used in the proof of Theorem \ref{thm:R0-4} and Theorem \ref{thm:R5-6}
to deal with the cases, say, $7\le n\le 40$, by hand. See Section \ref{sec:R40}.

\section{Some general results}\label{sec:general}
In this section, we will provide a few general statements concerning the MC (resp. non-MC) numbers and their corresponding MC-radii (resp. LUBORs). They will provide us a theoretical foundation for our computer search to be carried out in the next section. Recall that a lattice circle has at least three lattice points on its circumference.

\begin{thm}\label{thm:MCLatticeCircle}
If $n$ is MC then its MC-circle $\calM_n$ must be a lattice circle.
\end{thm}
\begin{proof}
Indeed, if $\calM_n$ has at most two lattice points on its circumference then we can continuously increase its size while keeping both these  point on $\calM_n$ and the number of interior lattice points unchanged. This contradicts the maximality of $\calM_n$.
\end{proof}

\begin{thm}\label{thm:RnIncrease}
Let $\{R_n: n\ge 0\}$ be the positive sequence of MC-radii (for MC $n$) and LUBORs (for non-MC $n$). Then   $\{R_n: n\ge 0\}$ is a non-decreasing sequence.
\end{thm}
\begin{proof}
We prove that $R_n\le R_{n+1}$ for any integer $n\ge 0$.

Choose $C_n$ to be the large $n$-enclosing circle $\calM_n$ if $n$ is MC by Theorem \ref{thm:MCLatticeCircle} and let $r=R_n$. If $n$ is non-MC then let $C_n$ be any circle of radius $r<R_n$ such that $C_n$ encloses $n$ points. In the latter case, by continuously deforming $C_n$ while keeping $r$ nondecreasing we may assume $C_n$ contains at least one lattice point $A$ on its circumference. Indeed, if there is no lattice point on $C_n$ then let $A$ be one of the lattice points outside $C_n$ that is closest to the center of the circle with the radial distance to the its boundary equal to $d$. Next, we enlarge the radius of $C_n$ by $2d/3$ and then move the circle towards $A$ by $d/3$ so that $A$ lies on the circumference of $C_n$ while the number of interior lattice points is unchange. This is possible since the smallest distance of any point outside $C_n$ is at least $d$, and an increasing of radius by $2d/3$ followed by a move of the circle center by a distance $d/3$ will not touch the point unless the final move is directly towards it.

We next show that we can deform $C_n$ by an infinitesimal amount $\vep>0$ and then move it to include one more point in its interior.

Let $A$ be a lattice point on $C_n$. While keeping $A$ on the circumference of $C_n$ we may pull the center $O$ of $C_n$ towards $A$ by an infinitesimal amount $\vep>0$ so that $A$ becomes the only lattice point on the new circle $C_n(\vep)$ and the number of lattice points inside $C_n(\vep)$ is still equal to $n$. A tiny perturbation of center of $C_n(\vep)$ along the same direction while keeping the same radius $r-\vep$ unchanged will now include $A$ in its interior.

This clearly implies that $R_{n+1}\ge r-\vep$ by the definition of $R_{n+1}$ either as an MC-radius or a LUBOR. Since $r=R_n$ if $n$ is MC and $r<R_n$ but can be arbitrarily close to $R_n$ if $n$ is non-MC, by taking $\vep\to 0$ we arrive at the desired conclusion $R_{n+1}\ge R_n$.
\end{proof}

\begin{thm}\label{thm:Non-MCLatticeCircle}
Let $n$ be non-MC and let $R_n$ be its LUBOR. Then there is a lattice circle centered at $\calO\in\Delta$ with radius $R_n$ and a sequence of circles centered at $\calO_j\in\Delta$ of radius $\rho_j$ ($j\ge 1$) enclosing the same set of $n$ lattice points such that the limit $\lim_{j\to \infty} \rho_j=R_n$.
\end{thm}
\begin{proof}
By rigid motions, we may assume without loss of generality that all circles we consider have centers in $\Delta$. It is clear that there are infinitely many $n$-enclosing circles by continuity argument. By the definition of $R_n$, there exists a sequence of such circles whose radii approach $R_n$ in the limit. Moreover, there are only finitely many different configurations of $n$ lattice point sets inside these circles. Therefore, there is a subsequence of infinitely many such circles whose radii approach $R_n$ in the limit and which contain the \textbf{same} set of $n$ lattice points. Suppose on the contrary the radius $R_n$ circle is not a lattice circle, then the same argument as used in the proof of Theorem \ref{thm:MCLatticeCircle} implies that $R_n$ is not an upper bound of the radii of all $n$-enclosing circles, which contradicts the assumption.
\end{proof}

\begin{rem}
Notice that the circle $\calO_j$ in Theorem \ref{thm:Non-MCLatticeCircle} cannot be lattice circles for all sufficiently large $j$ by discreteness of lattice circles. And the circle $\calO$ must not be the largest $n$-enclosing lattice circle; otherwise, it would contradicts to the condition that $n$ is non-MC.
\end{rem}

The next corollary quickly follows from Theorems \ref{thm:MCLatticeCircle} and \ref{thm:Non-MCLatticeCircle}.
\begin{cor}\label{cor:RnAllLatticeCircleRadius}
For every nonnegative integer $n$, either as an MC-radius or as a LUBOR, $R_n$ is the radius of some lattice circle.
\end{cor}

\begin{thm}\label{thm:FindNonMaxRadius}
Suppose $k<\ell$ are two consecutive MC numbers and $n$ is non-MC such that $k<n<\ell$. Then $R_n=R_k$, namely, $R_k$ is the LUBOR for all such non-MC $n$ immediately following $k$.
\end{thm}

\begin{proof}
Let $n=k+j$ be a non-MC given as in the theorem, $j=1,\dots,\ell-k-1$. We now prove the theorem by induction on $j$.

If $j=1$, by definition of the LUBOR $R_n$, if $R_n>R_k$ then there is some $r>R_k$ such that a circle of radius $r$ encloses exactly $n$ lattice points. By Theorem \ref{thm:Non-MCLatticeCircle}, $R_n$ is the radius of some lattice circle $\calO$ and there is a sequence of circles $\calO_j$ enclosing the same set $S$ of $n$ lattice points with radii $\rho_j\to R_n$. Notice that in the limit the circle $\calO$ cannot enclose any other points than $S$ but it might lose some in $S$ to its boundary. Therefore, the number of lattice points in $\calO$, say $m$, must be no more than $n$. Since $n$ is non-MC, $\calO$ cannot contain exactly $n$-points, that is, $m\le n-1=k$. Thus its radius $R_n\le R_m\le R_k$ by the definition of $R_m$ and Theorem \ref{thm:RnIncrease}. But $R_n\ge R_k$ by the same theorem as $n>k$ and therefore $R_n=R_k$.

Suppose $j>1$ and suppose $R_{k+1}=\cdots=R_{k+j-1}=R_k$. Then by the same argument as above we see that $R_n\le R_m$ for some $m\le n-1$. By induction assumption and Theorem \ref{thm:RnIncrease}, we see immediately that
$R_n\le R_k$. As we already know $R_n\ge R_k$ by Theorem \ref{thm:RnIncrease} since $n>k$, we must have $R_n= R_k$. This completes the proof of the theorem.
\end{proof}

\begin{cor}\label{cor:NonMCradiiStrictDecrease}
Suppose $n>k$ and $k$ is an MC number and all the numbers between $k$ and $n$ (exclusive) are non-MC. Then the radius of the largest $n$-enclosing \textbf{lattice} circle is \textbf{strictly} smaller than $R_k$ if and only if $n$ is non-MC.
\end{cor}
\begin{proof}
Suppose $C$ is the largest $n$-enclosing \textbf{lattice} circle with radius $r$. If $n$ is MC then $C$ is an MC-circle of $n$ and therefore its radius $r=R_n\ge R_k$ by Theorem \ref{thm:RnIncrease}. On the other hand, if $n$ is non-MC then $r<R_n=R_k$ by Theorem \ref{thm:FindNonMaxRadius}.
\end{proof}

\begin{defn}
For any MC number $n$, the length of the longest consecutive non-MC number sequence immediately following $n$ is called its \emph{impacting index}, denoted by $I_n$. We call an MC number \emph{strong} if $I_n>0$.
\end{defn}

\begin{eg}
For example, 4 is a strong MC number with $I_4=2$ since both 5 and 6 are non-MC.  We have seen in Theorem \ref{thm:FindNonMaxRadius} that strong MC numbers play the pivotal role in determining the LUBORs of non-MC numbers.
\end{eg}

\section{More complicated cases: $19\le n\le 40$}\label{sec:R40}

The following theorem covers the cases $19\le n\le 26$.
\begin{thm}\label{thm:R19-26}
We have the following:
\begin{figure}[H]
\centering
  \includegraphics[scale=0.5]{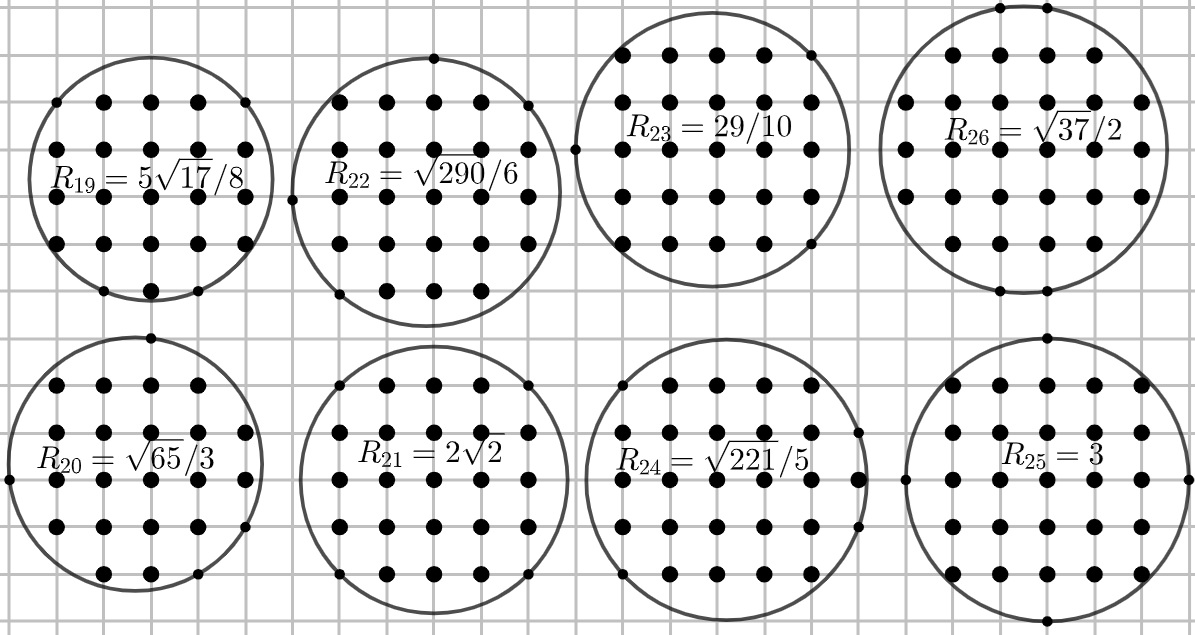}
\caption{The largest circles enclosing 19 to 26 lattice points.}
\label{fig:R19-26}
\end{figure}
\end{thm}
\begin{proof}
We only need to prove each number in Figure~\ref{fig:R19-26} is an upper bound of the corresponding $R_n$.

$n=19$. By the left picture of Figure~\ref{fig:R19Proof}, we see easily that $O$ encloses the 18 points $L_j, 1\le j\le 18$ if its radius is greater than $5\sqrt{17}/8$. By the same argument in case $n=13$, either $X'$ or $Y'$ is in the interior of the circle. Now to find the 20th point, we need to consider the right picture of Figure~\ref{fig:R19Proof} in which the triangle $\triangle XYZ$ has circumradius $5\sqrt{13}/7<5\sqrt{17}/8$. Since the triangle $\triangle CML_6\subseteq \triangle XYZ$ we see immediately that at least one of the distances $|OX|$, $|OY|$ or $|OZ|$ is at most $5\sqrt{13}/7<5\sqrt{17}/8$. This implies that either $X$, or $Y$, or $Z$ lies inside $O$. Thus $O$ encloses at least 20 points.

\begin{figure}[H]
\centering
  \includegraphics[scale=0.8]{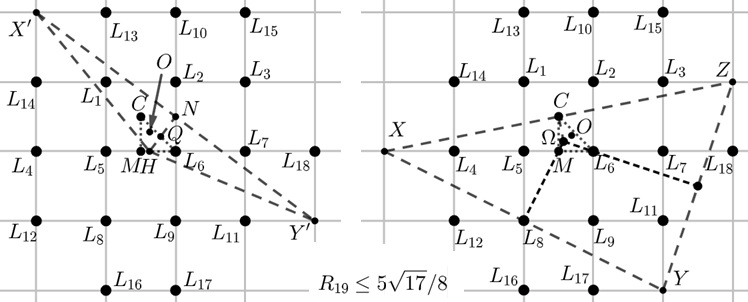}
\caption{Upper bound of $R_{19}$.}
\label{fig:R19Proof}
\end{figure}

$n=20$. By the three pictures of Figure~\ref{fig:R20Proof}, we see easily that $O$ encloses the 18 points $L_j, 1\le j\le 18$ if its radius is greater than $\sqrt{65}/3$. By the same argument in case $n=13$, either $X_k$ or $Y_k$ is in the interior of the circle for each of $k=1,2,3$. Thus $O$ encloses at least 21 points.

\begin{figure}[H]
\centering
  \includegraphics[scale=0.8]{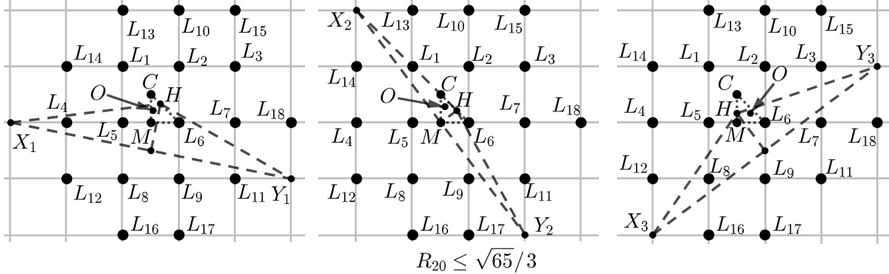}
\caption{Upper bound of $R_{20}$.}
\label{fig:R20Proof}
\end{figure}

$n=21$. By the left two pictures of Figure~\ref{fig:R21-22Proof}, we see easily that $O$ encloses the 20 points $L_j, 1\le j\le 20$ if its radius is greater than $2\sqrt{2}$. By the same argument in case $n=20$, either $X_k$ or $Y_k$ is in the interior of the circle for each of $k=1,2$. Thus $O$ encloses at least 22 points.

\begin{figure}[H]
\centering
  \includegraphics[scale=0.54]{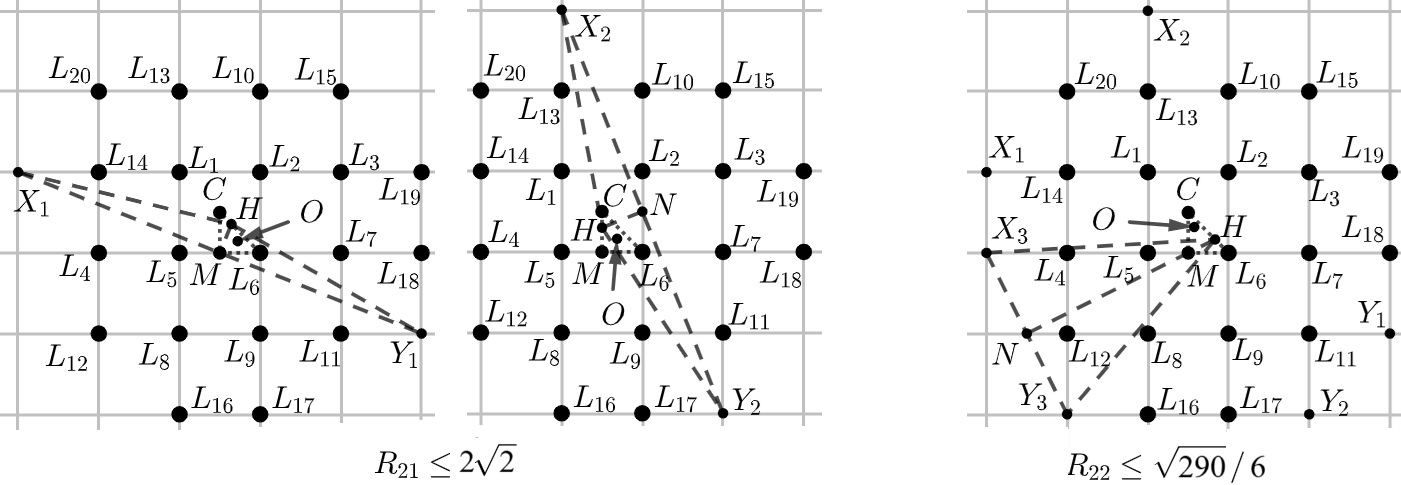}
\caption{Upper bound of $R_{21}$ and $R_{22}$.}
\label{fig:R21-22Proof}
\end{figure}

$n=22$ and $n=23$. When $n=21$, we have seen that $O$ encloses the 20 points $L_j, 1\le j\le 20$ if its radius is greater than $\sqrt{290}/6>2\sqrt{2}$, and either $X_k$ or $Y_k$ is in the interior of the circle for each of $k=1,2$. If $n=22$, by the third picture of Figure~\ref{fig:R21-22Proof}, we can check that either $X_3$ or $Y_3$ is in the interior of the circle. Thus $O$ encloses at least 23 points.
If $n=23$, by the two pictures of Figure~\ref{fig:R23Proof}, either $X_k$ or $Y_k$ is in the interior of the circle for each of $k=3,4$. Thus $O$ encloses at least 24 points.

\begin{figure}[H]
\centering
  \includegraphics[scale=0.5]{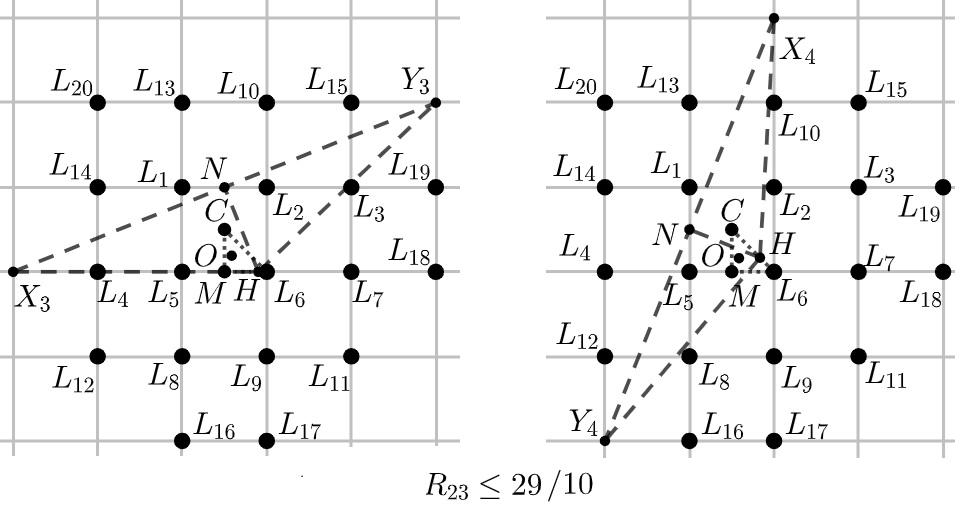}
\caption{Upper bound of $R_{23}$.}
\label{fig:R23Proof}
\end{figure}

$n=24$. It clear in the left picture of Figure~\ref{fig:R24-25Proof} that $O$ encloses the 23 points $L_j, 1\le j\le 23$ if its radius is greater than $\sqrt{221}/5$, and either $X_k$ or $Y_k$ is in the interior of the circle for each of $k=1,2$. Thus $O$ encloses at least 25 points.

\begin{figure}[H]
\centering
  \includegraphics[scale=0.5]{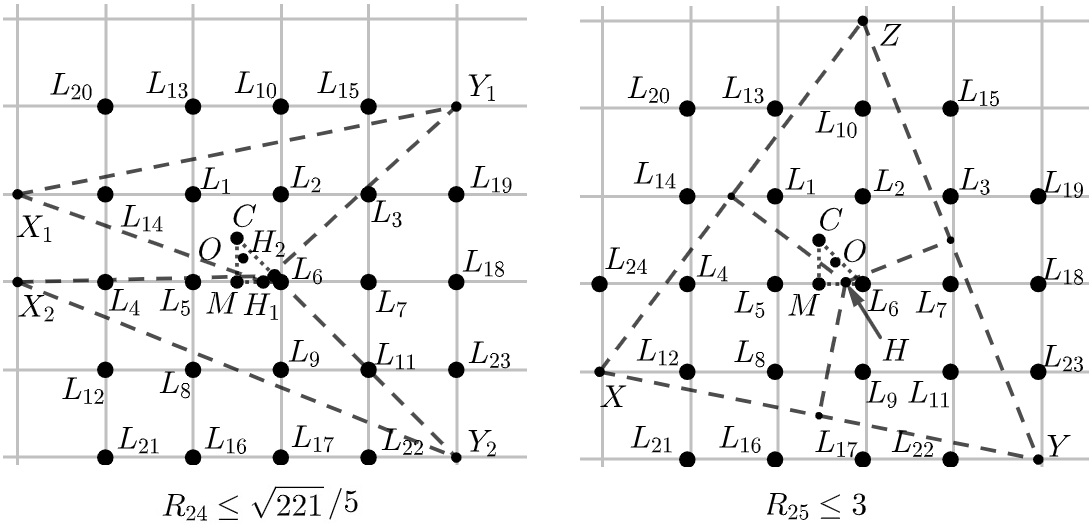}
\caption{Upper bound of $R_{24}$ and $R_{25}$.}
\label{fig:R24-25Proof}
\end{figure}

$n=25$. In this case, a \emph{new idea} must be applied for the first time in the paper. We see in the right picture of Figure~\ref{fig:R24-25Proof} that $O$ encloses the 24 points $L_j, 1\le j\le 24$ if its radius is greater than $3$, and from the proof in case $n=24$, either $X_1$ or $Y_1$ is in the interior of the circle. But where is the 26th point? It turns out the circumradius of $\triangle XYZ$ in Figure~\ref{fig:R24-25Proof} is $5\sqrt{754}/46<3$ (and incidentally, the circumcenter $H$ is not on the grid line albeit very close). Therefore, $O$ must enclose one of three points $X$, $Y$, or $Z$. In summary, $O$ encloses at least 26 points.

\begin{figure}[H]
\centering
  \includegraphics[scale=0.5]{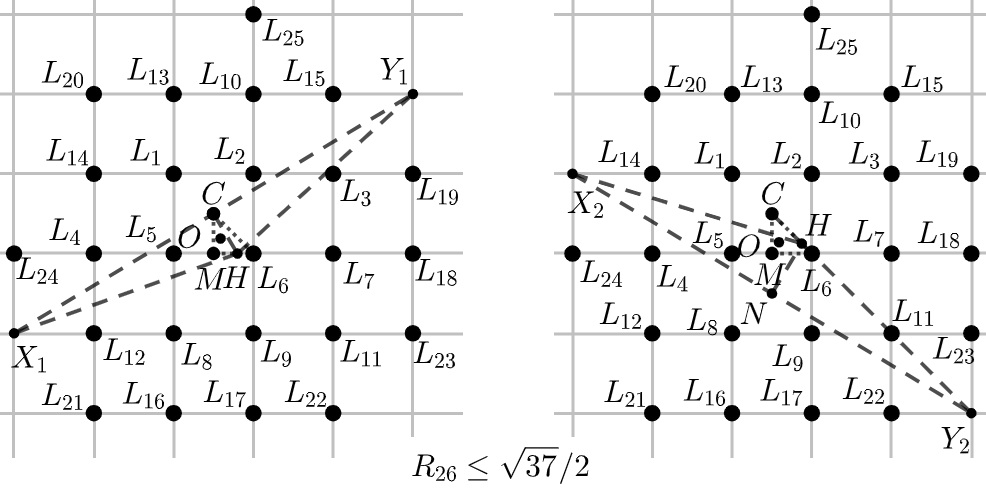}
\caption{Upper bound of $R_{26}$.}
\label{fig:R26Proof}
\end{figure}
$n=26$. It clear in Figure~\ref{fig:R26Proof} that $O$ encloses the 25 points $L_j, 1\le j\le 25$ if its radius is greater than $\sqrt{37}/2$, and either $X_k$ or $Y_k$ is in the interior of the circle for each of $k=1,2$. Thus $O$ encloses at least 27 points.
\end{proof}

The following theorem covers the cases $27\le n\le 32$.

\begin{thm}\label{thm:R27-32}
We have the following:
\begin{figure}[H]
\centering
  \includegraphics[scale=0.5]{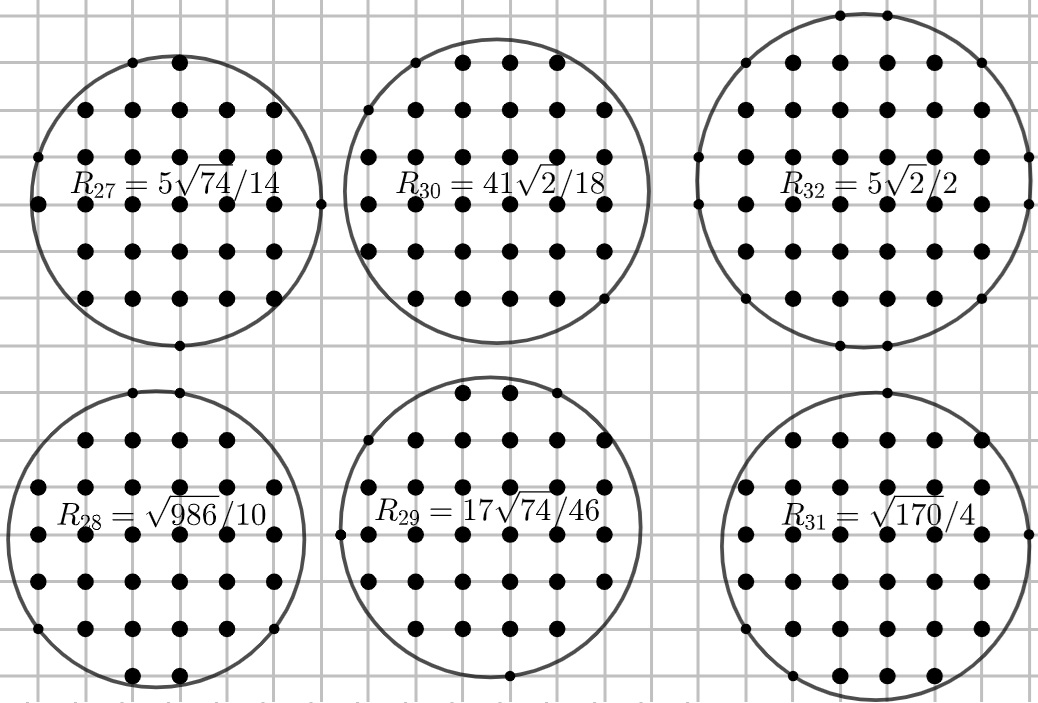}
\caption{The largest circles enclosing 27 to 32 lattice points.}
\label{fig:R27-32}
\end{figure}
\end{thm}

\begin{proof}

$n=27$. It clear in Figure~\ref{fig:R26Proof} that $O$ encloses the 25 points $L_j, 1\le j\le 25$ if its radius is greater than $5\sqrt{74}/14>\sqrt{37}/2$. Furthermore, by Figure~\ref{fig:R27Proof}, either $X_k$ or $Y_k$ is in the interior of the circle for each of $k=1,2$.  And the circumradius $R=\sqrt{6290}/26\approx 3.05$ of $\triangle XYZ$ satisfies $R<5\sqrt{74}/14\approx 3.07$.Thus $O$ encloses at least 28 points.

\begin{figure}[H]
\centering
  \includegraphics[scale=0.56]{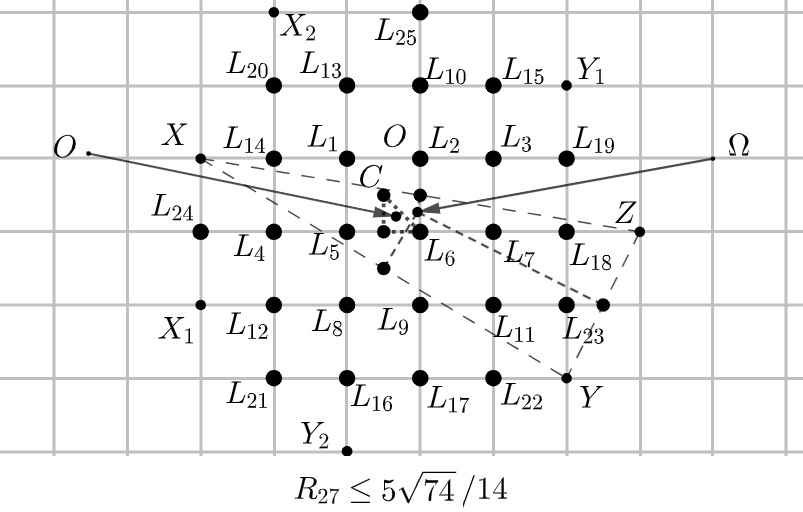}
\caption{Upper bound of $R_{27}$.}
\label{fig:R27Proof}
\end{figure}

$n=28$.
It clear in Figure~\ref{fig:R26Proof} that $O$ encloses the 25 points $L_j, 1\le j\le 25$ if its radius is greater than $\sqrt{85}/3>\sqrt{37}/2$, and either $X_1$ or $Y_1$ is in the interior of the circle. Now by Figure~\ref{fig:R28Proof} either $X_k$ or $Y_k$ is in the interior of the circle for each of $k=2,3,4$. Thus $O$ encloses at least 29 points.

\begin{figure}[H]
\centering
  \includegraphics[scale=0.75]{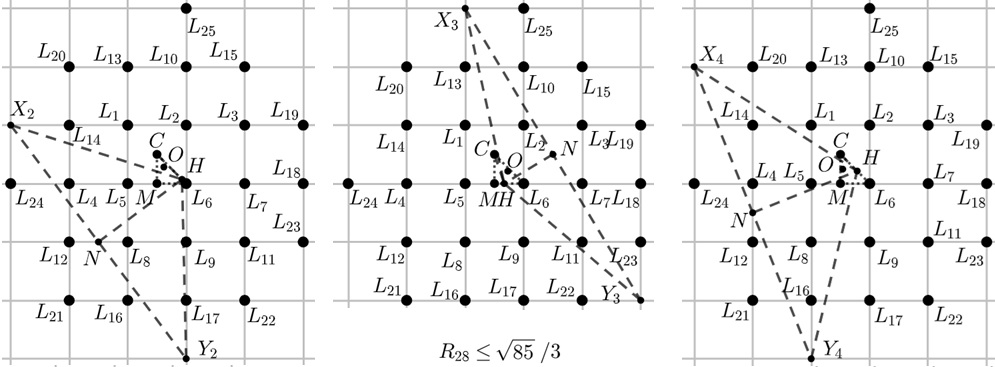}
\caption{Upper bound of $R_{28}$.}
\label{fig:R28Proof}
\end{figure}

$n=29$.
By Figure~\ref{fig:R29Proof}, $O$ encloses the 28 points $L_j, 1\le j\le 28$ if its radius is greater than $17\sqrt{74}/46$, and either $X_1$ or $Y_1$ is in the interior of the circle. Furthermore, the circumradius of $\triangle XYZ$ in Figure~\ref{fig:R29Proof} is $17\sqrt{74}/46$ so one of the three vertices $X$, $Y$ and $Z$ is in  the interior of the circle. Thus $O$ encloses at least  30 points.

\begin{figure}[H]
\centering
  \includegraphics[scale=0.5]{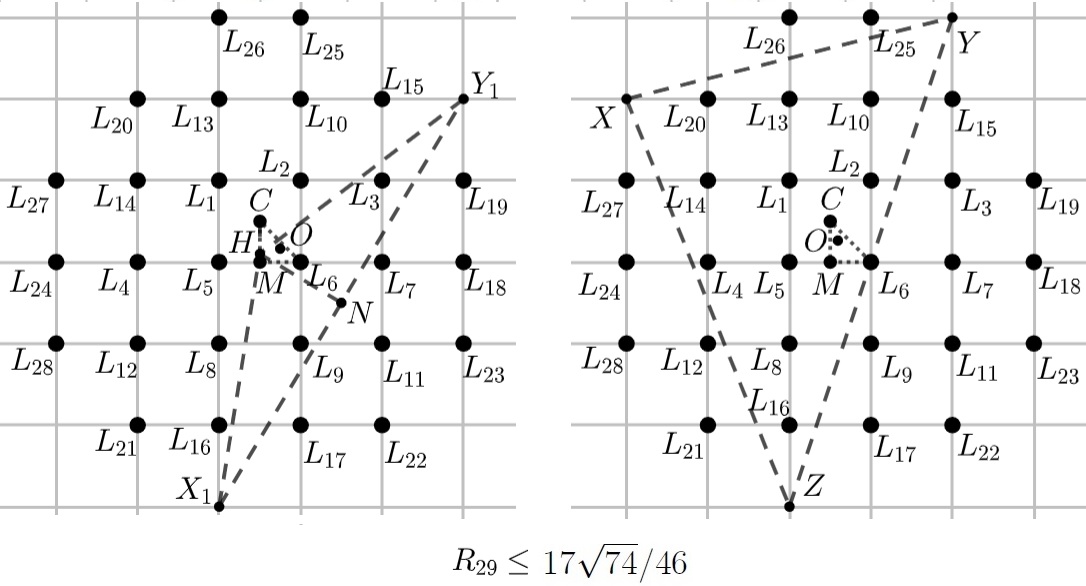}
\caption{Upper bound of $R_{29}$.}
\label{fig:R29Proof}
\end{figure}

$n=30$.
It clear in Figure~\ref{fig:R29Proof} that $O$ encloses the 28 points $L_j, 1\le j\le 28$ if its radius is greater than $41\sqrt{2}/18>17\sqrt{74}/46$, and either $X_1$ or $Y_1$ is in the interior of the circle. Now by the left picture of Figure~\ref{fig:R30Proof} either $X_2$ or $Y_2$ is in the interior of the circle. By the right picture of Figure~\ref{fig:R30Proof} the circumradius of $\triangle XYZ$ in Figure~\ref{fig:R29Proof} is $41\sqrt{2}/18$ so one of the three vertices $X$, $Y$ and $Z$ is in  the interior of the circle. Thus $O$ encloses at least 29 points.

\begin{figure}[H]
\centering
  \includegraphics[scale=0.7]{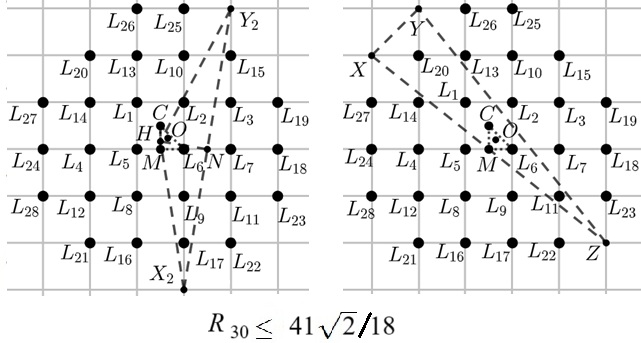}
\caption{Upper bound of $R_{30}$.}
\label{fig:R30Proof}
\end{figure}

$n=31$.
It clear in Figure~\ref{fig:R31Proof} that $O$ encloses the 29 points $L_j, 1\le j\le 29$ if its radius is greater than $\sqrt{170}/4$, and either $X_k$ or $Y_k$ is in the interior of the circle for each of $k=1,2,3$. Thus $O$ encloses at least 32 points.

\begin{figure}[H]
\centering
  \includegraphics[scale=0.55]{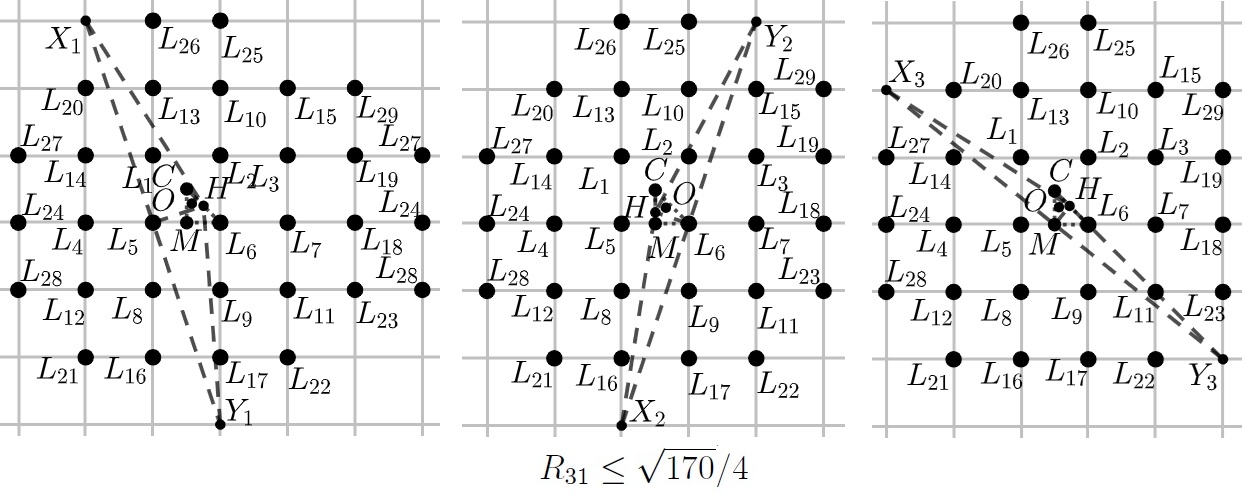}
\caption{Upper bound of $R_{31}$.}
\label{fig:R31Proof}
\end{figure}

$n=32$.
By Figure~\ref{fig:R32Proof} $O$ encloses the 34 points $L_j, 1\le j\le 34$ if its radius is greater than $5\sqrt{2}/2$, and either $X_k$ or $Y_k$ is in the interior of the circle for each of $k=1,2,3$. Thus $O$ encloses at least 37 points.
\end{proof}

\begin{figure}[H]
\centering
  \includegraphics[scale=0.5]{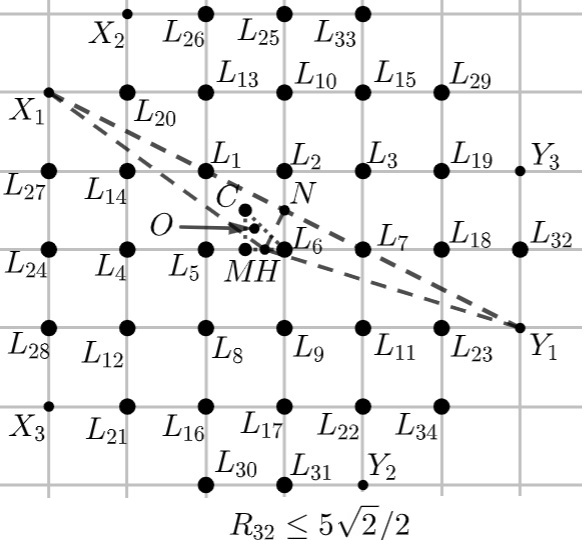}
\caption{Upper bound of $R_{32}$.}
\label{fig:R32Proof}
\end{figure}

It turns out that $n=33,34,35,36$ are all non-MC.
\begin{lem}\label{lem:n=33Step2}
For any real number $r\ge5\sqrt{2}/2$, every circle of radius $r$ must
enclose exactly $32$ lattice points or at least $37$ lattice points.
\end{lem}

\begin{proof}
This is clear from the proof of case $n=32$ in the above theorem when $r>5\sqrt{2}/2$. If $r=5\sqrt{2}/2$ and
the center of the circle is not at the center of any unit square, then the same proof works.
\end{proof}

\begin{thm}\label{thm:R33-36}
We have $R_n\prec 5\sqrt{2}/2$ for $33\le n\le 36$.
\end{thm}
\begin{proof}
For each $33\le n\le 36$, if $n$ were MC then $R_n\ge R_{32}$ by Theorem \ref{thm:RnIncrease}. This contradicts Lemma \ref{lem:n=33Step2} if $R_n>R_{32}$ and therefore $R_n=R_{32}$.
\end{proof}

\begin{thm}\label{thm:R37}
We have $R_{37}=\sqrt{13}$ and $R_{38}\prec \sqrt{13}$.
\end{thm}

\begin{proof}
$n=37$.
By left picture of Figure~\ref{fig:R37Proof}, circle $O$ with radius $\sqrt{13}$ encloses the 37 points $L_j, 1\le j\le 37$ and therefore  $R_{37}\ge \sqrt{13}$. If its radius is greater than $\sqrt{13}$, then by the right picture of Figure~\ref{fig:R37Proof} $O$ and encloses the 38 points $L_j, 1\le j\le 38$ and either $X$ or $Y$ in the interior of the circle. Thus $O$ encloses at least 39 points.

\begin{figure}[H]
\centering
  \includegraphics[scale=0.5]{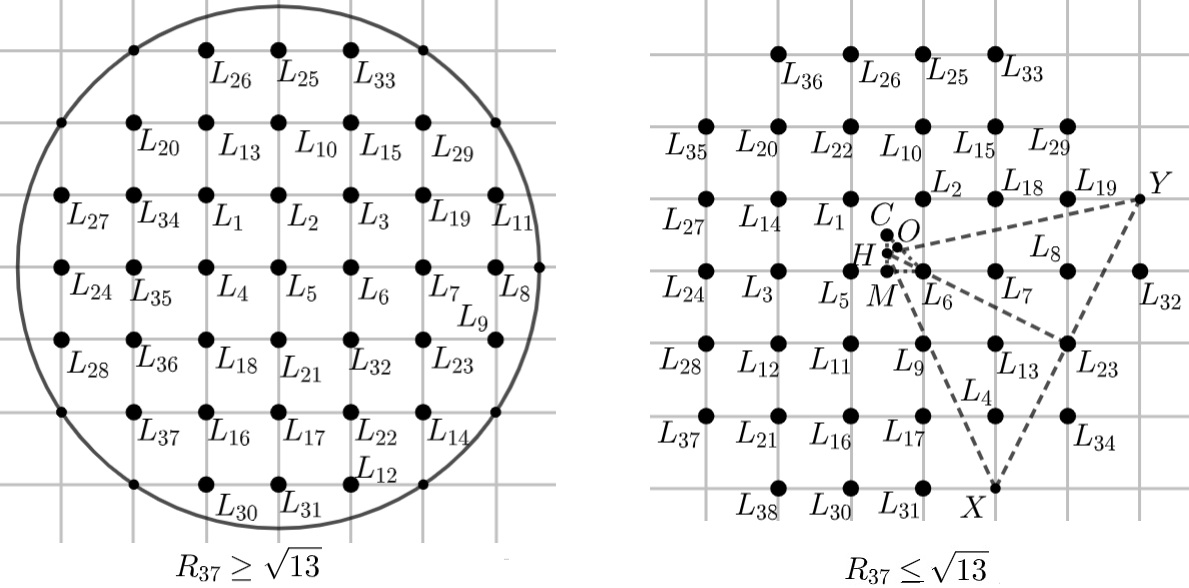}
\caption{Upper bound of $R_{37}$.}
\label{fig:R37Proof}
\end{figure}

Theorem \ref{thm:RnIncrease} and the above argument imply that 38 is non-MC and therefore $R_{38}\prec \sqrt{13}$ follows from
Theorem~\ref{thm:FindNonMaxRadius}.
\end{proof}

\begin{thm}\label{thm:R39}
We have $R_{39}=\frac{5\sqrt{34}}{8}$ and $R_{40}=\frac{\sqrt{481}}{6}$.
\end{thm}

\begin{proof}
$n=39$.
By left picture of Figure~\ref{fig:R39Proof}, $O$ with radius $\frac{5\sqrt{34}}{8}$ encloses the 39 points $L_j, 1\le j\le 39$ and therefore  $R_{39}\ge\frac{5\sqrt{34}}{8}$. If its radius is greater than $\frac{5\sqrt{34}}{8}$, then by the right picture of Figure~\ref{fig:R39Proof} $O$ encloses the 38 points $L_j, 1\le j\le 38$ and either $X$, or $Y$ or $Z$ because the circumradius of $\triangle XYZ$ is $\sqrt{15170}/34\approx 3.62<\frac{5\sqrt{34}}{8}\approx 3.644344934$. It further encloses either $X_1$ or $Y_1$. Thus $O$ encloses at least 40 points. It implies that $R_{39}=\frac{5\sqrt{34}}{8}$.

\begin{figure}[H]
\centering
  \includegraphics[scale=0.5]{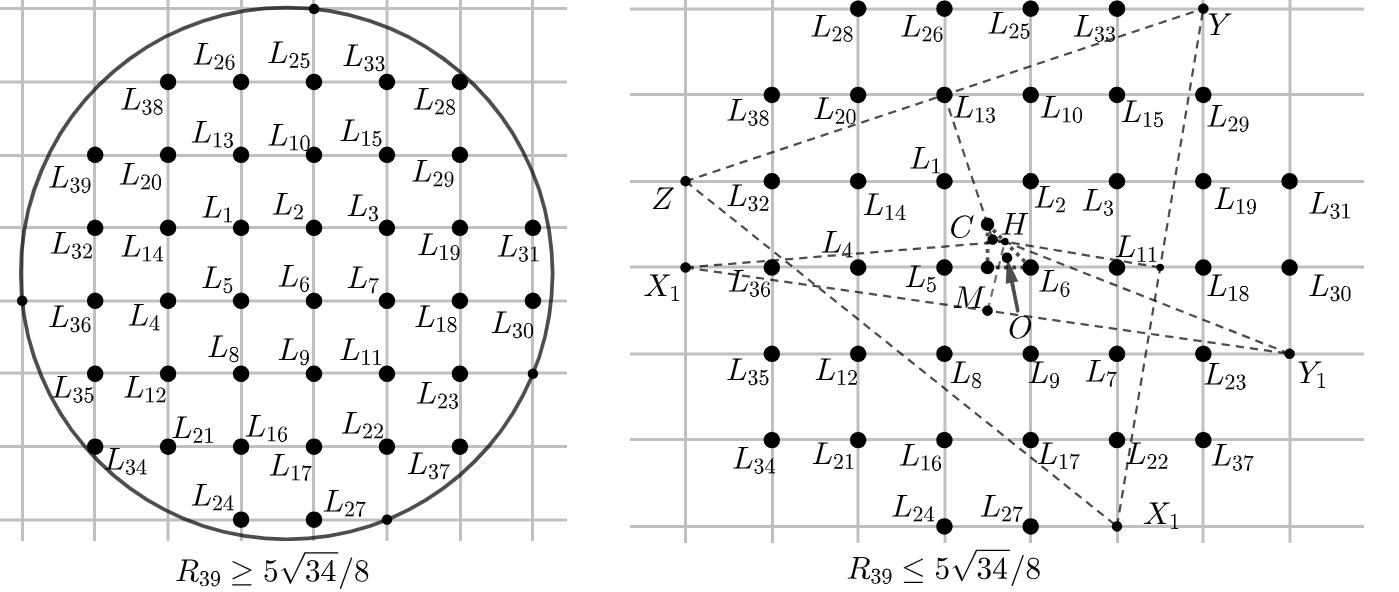}
\caption{Upper bound of $R_{39}$.}
\label{fig:R39Proof}
\end{figure}

$n=40$.
By left picture of Figure~\ref{fig:R40Proof}, $O$ with radius $\frac{\sqrt{481}}{6}$ encloses the 39 points $L_j, 1\le j\le 40$ and therefore  $R_{40}\ge\frac{\sqrt{481}}{6}$.
\begin{figure}[H]
\centering
  \includegraphics[scale=0.5]{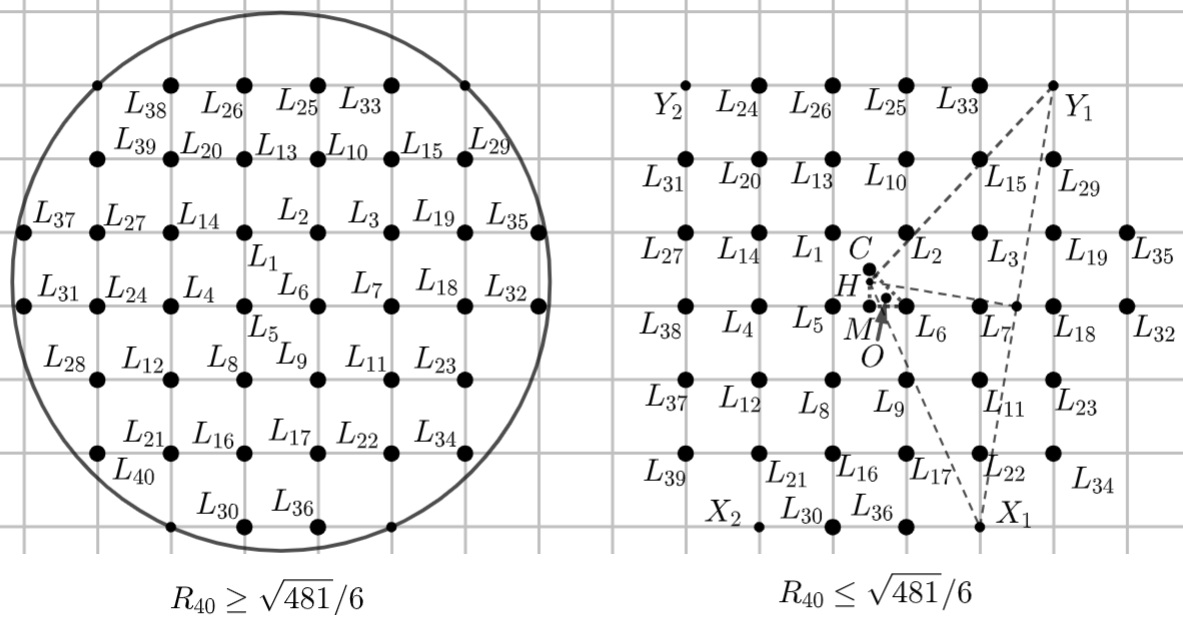}
\caption{Upper bound of $R_{40}$.}
\label{fig:R40Proof}
\end{figure}

If its radius is greater than $\frac{\sqrt{481}}{6}$, then by the right picture of Figure~\ref{fig:R40Proof} $O$ encloses the 39 points $L_j, 1\le j\le 39$ either $X_1$ or $Y_1$ by easy computation. By symmetry, it further encloses either $X_2$ or $Y_2$. Thus $O$ encloses at least 41 points. Therefore, $R_{40}=\frac{\sqrt{481}}{6}$.
\end{proof}

We now summarize our computation for $n\le 40$ in the following table.

\begin{table}[H]
\caption{Radius of largest $n$-enclosing circle for MC $n$, or its least upper bound for non-MC $n$.}
\centering\begin{tabular}{|c|c|c|c|c|c|c|c|c|c|c||c|c|}
\hline
\boldmath{$n$} & 1 &  2          &  3            &  4           &  5, 6               & 7    &  8            &9      &10 &11    &  12   &   13        \\
\hline $R_n$ & 1 &$\frac{\sqrt{5}}{2}$ &$\frac{5\sqrt{2}}{3}$ &$\frac{\sqrt{10}}{2}$ & $\prec R_4$ &$\frac53$ &$\frac{\sqrt{13}}{2}$ &2 &$\frac{\sqrt{65}}{4}$ &$\frac{\sqrt{442}}{10}$ &$\frac{17}{8}$ &$\sqrt{5}$ \\
\hline
\boldmath{$n$} &     14       &     15       &     16        &  17,18       &     19       &     20       &     21       &     22&     23       &     24 &25 &26\\
\hline $R_n$ &$\frac{\sqrt{85}}{4}$ &$\frac{5\sqrt{2}}{3}$ &$\frac{\sqrt{26}}{2}$ &$\prec R_{16}$ &$\frac{5\sqrt{17}}{8}$ &$\frac{\sqrt{65}}{3}$ &$2\sqrt{2}$ &$\frac{\sqrt{290}}{6}$ &$ \frac{29}{10}$ &$ \frac{\sqrt{221}}{5}$ &$3$  &$\frac{\sqrt{37}}{2}$  \\
\hline
\boldmath{$n$} &   27       & 28       &    29       &  30       &   31      &     32       &     33-36     &    37     &  38 &39 & 40 & \  \\
\hline $R_n$ &$\frac{\sqrt{37}}{2}$ &$\frac{\sqrt{85}}{3}$ &$\frac{17\sqrt{74}}{46}$ &$\frac{41\sqrt{2}}{18}$ &$\frac{5\sqrt{2}}{2}$ &$\frac{5\sqrt{2}}{2}$ &$\prec R_{32}$ &$\sqrt{13}$ &$\prec R_{37}$   &$\frac{5\sqrt{34}}{8} $ &$\frac{\sqrt{481}}{6}$  &\  \\
\hline
\end{tabular}
\end{table}

\section{Two special classes of $n$-enclosing circles}\label{sec:SpecialClass}

\subsection{First special class}\label{subsec:1stClass}
For any positive integer $k$, we may define the unique circle $\calS_k$
that goes through the following points: $A_1(-k,0),A_2(-k,1), A_3(0,k+1),
A_4(1,k+1), A_5(k+1,1), A_6(k+1,0), A_7(1,-k), A_8(0,-k).$ Suppose
$\calS_k$ encloses exactly $f(k)$ lattice points, then $\calS_k$ is a possible candidate for the largest
$f(k)$-enclosing circle. If $\calS_k$ is indeed the largest $f(k)$-enclosing circle then $f(k)$ is MC and
$$R_{f(k)}=\sqrt{k^2+k+\frac12}.$$
We have already seen this is true when $k=1$ and $k=2$ in the case $n=4$ in Theorem \ref{thm:R0-4}. We depicted these circles (and many others) in Figure \ref{fig:impactingFactorAll} on page \pageref{fig:impactingFactorAll} for $k\le 18$, $k\ne 4,7,12, 16$.

In general, it is not hard to see that
\begin{equation*}
    f(k)=4\sum_{j=1}^k \left(\left\lceil \sqrt{k^2 + k +  \frac12 - \Big(j - \frac12\Big)^2} +  \frac12 \right\rceil - 1 \right) \sim \pi  R_{f(k)}^2=\pi\left(k^2+k+\frac12\right).
\end{equation*}
To compare these values, we have the following table:
\begin{table}[H]
\caption{First special class of circles with apparent symmetric lattice points on each boundary.}
\label{tab:specialClass}
\centering\begin{tabular}{|c|c|c|c|c|c|c|c|c|c|c|}
\hline
$k$ &        1 & 2  & 3 & 4 & 5 & 6 & 7 & 8 & 9   &  10 \\
\hline
$f(k)$ &   4 & 16  & 32 & 60 & 88 & 124 & 172 & 216 & 276 &   332 \\
\hline
$\pi  R_{f(k)}^2\approx $ & 8 & 20 & 39 & 64 & 96 & 134 & 176 & 228 & 284  & 347   \\
\hline
$k$         &       11 & 12  & 13 & 14 & 15 & 16 & 17 & 18  &19 &20\\
\hline
$f(k)$ & 408 & 484  & 560 & 648 & 740 & 848 & 952 & 1060  &1184 & 1304\\
\hline
$\pi  R_{f(k)}^2\approx $ &  416& 492& 573& 661& 756& 856& 963& 1076 & 1195 &1321 \\
\hline
\end{tabular}
\end{table}

\subsection{Second special class}\label{subsec:2ndClass}
For any positive integer $k$, we can mimic the first class by letting $\calT_k$ be the circle centered at $(0,0)$ containing the eight lattice points $(\pm k,\pm 1),(\pm 1,\pm k)$. Suppose
$\calT_k$ encloses exactly $g(k)$ lattice points, then $\calT_k$ is a possible candidate for the largest
$g(k)$-enclosing circle. If $\calT_k$ is indeed the largest $g(k)$-enclosing circle then $g(k)$ is MC and
$$R_{g(k)}=\sqrt{k^2+1}.$$
Similar to the first class in Subsection \ref{subsec:1stClass}, we have
\begin{equation*}
   g(k)\sim \pi  R_{g(k)}^2=\pi (k^2+1).
\end{equation*}
To compare these values, we have the following table:
\begin{table}[H]
\caption{Second special class of circles with apparent symmetric lattice points on each boundary.}
\label{tab:specialClass2}
\centering\begin{tabular}{|c|c|c|c|c|c|c|c|c|c|c| }
\hline
$k$ &        1 & 2  & 3 & 4 & 5 & 6 & 7 & 8 & 9  &  10  \\
\hline
$g(k)$ &   5& 13& 29& 49& 81& 113& 149& 197& 253 & 317\\
\hline
$\pi  R_{g(k)}^2\approx $ & 6& 16& 31& 53& 82& 116& 157& 204& 258 & 317   \\
\hline
$k$         &   11 & 12  & 13 & 14 & 15 & 16 & 17 & 18  &19 &20  \\
\hline
$g(k)$ &    377& 441& 529& 613& 709& 797& 901& 1009 &1129 &1257\\
\hline
$\pi  R_{g(k)}^2\approx $ & 383&  456& 534& 619& 710& 807& 911& 1021 & 1137 & 1260 \\
\hline
\end{tabular}
\end{table}
Among all the $g(k)$'s in Table \ref{tab:specialClass2} with $g(k)<1100$ the following are all strong MC: $k=4,7,8,11,12,16,17,18$. We will present their MC-circle symmetry in Figure \ref{fig:impactingFactorAll} on page \pageref{fig:impactingFactorAll}.

Notice that the sequence $\{g(k)\}_{k\ge 1}$ is [A000328] on the OEIS website \cite{Sloane2025}. This is equivalent to saying that when we shrink the circle $\calT_k$ to radius $k$, no interior points are moved outside. Indeed, if lattice point $(a,b)$ satisfies $a^2+b^2<k^2+1$ if and only if $a^2+b^2\le k^2$.

\subsection{Some observations of the two special classes}\label{sec:ObserveSpecialClasses}
Recall that an MC number is said to be strong if it is followed immediately by a non-MC number.
We noticed that there are many MC numbers in the two classes in subsections \ref{subsec:1stClass} and  \ref{subsec:2ndClass}, and some are even strong ones.

The first class of special circles $\calS_k$ seem to contain only MC numbers. We find the second class contain only two non-MC numbers 5 and 317 in the table.
However, these two classes seem to behave very differently. While each $\calS_k$ ($k\le 19$) is the largest $f(k)$-enclosing circle, most of the $\calT_k$'s are not. Among the 19 numbers, $\calT_k$ is the largest $g(k)$-enclosing circle only for $g(k)=13, 49, 149, 197, 377, 441, 797, 901, 1009$. From this, it seems that as $k$ becomes bigger, it is more likely that  $\calT_k$ is the largest $g(k)$-enclosing circle.
In view of the apparent symmetries among the lattice points on the circles $\calS_k$ and $\calT_k$, we make for following conjecture.

\begin{conj} \label{conj:SpecialClass}
Let $f(k)$ and $g(k)$ be defined as in Subsections \ref{subsec:1stClass} and  \ref{subsec:2ndClass}. Then
\begin{enumerate}[label=(\arabic*)]
  \item \label{conj:SpecialClassCase1} The special class of numbers $\{f(k)\}_{k\ge1}$ contain only MC numbers, and the strong numbers among them form a subset of positive density.

  \item \label{conj:SpecialClassCase2}
  The special class of numbers $\{g(k)\}_{k\ge1}$ contain infinitely many MC numbers, and the strong numbers among them form a subset of positive density.
\end{enumerate}
\end{conj}

Due to lack of more evidence, we also would like to know the answer to the following question.

\begin{prob}
Are there any other non-MC number beside 5 and 317 in the sequence $\{g(k)\}_{k\ge1}$?
\end{prob}

\section{Complete classification of $n\le 1000$}\label{sec:computer}
When $n$ gets larger and larger, it becomes more and more difficult to find the exact values of $R_n$. However, we have the following theoretical results to help us delegate the main computation to an algebra system such MAPLE or Mathematica. The first one is an analogous formula of \eqref{equ:NumberOfPt} for the number of interior points inside a circle of radius $r$.

\begin{thm}
For any circle of radius $r$ centered at the origin, the number $\nu(r)$  of interior lattice points is given by
\begin{equation}\label{equ:NumberOfIntPt}
     \nu(r)=4\lceil r \rceil-3 +4\sum_{j=1}^{\lceil r \rceil-1} (\lceil \sqrt{r^2-j^2}\rceil-1).
\end{equation}
\end{thm}
\begin{proof}
It is clear that the number of  interior lattice points on the $x$- or $y$-axes is
\begin{equation*}
    1+ 4(\lceil r \rceil-1)=4\lceil r \rceil-3
\end{equation*}
since  $\lceil r \rceil=r$ if $r$ is an integer while $\lceil r \rceil-1$ is the integer part of $r$ otherwise. Suppose now $(j,k)$ is in the interior of the circle with $j,k>0$. Then they must satisfy $1\le j\le \lceil r \rceil-1$ and $j^2+k^2<r^2$, whence the summands in \eqref{equ:NumberOfIntPt}.
\end{proof}

\begin{thm}\label{thm:RadiusBound}
Let $n$ be a non-negative integer. Suppose $R_n$ is the MC-radius or the LUBOR of $n$. Then $$\sqrt{n/\pi}<R_n<\sqrt{2}+\sqrt{n/\pi}.$$
\end{thm}
\begin{proof}
By Steinhaus' Theorem \ref{thm:Steinhaus}, there exists a circle of area $n$ which encloses exactly $n$ lattice points in its interior. Since its radius must be bounded above by $R_n$ we get $\pi R_n^2>n$. Note the equality cannot hold as $\pi$ is transcendental while $R_n$ is a rational or a quadratic algebraic number as a radius of some lattice circle by Theorem \ref{thm:MCLatticeCircle} and Theorem \ref{thm:Non-MCLatticeCircle}.

The other direction $R_n<\sqrt{2}+\sqrt{n/\pi}$ follows from Gauss' bound \eqref{equ:GaussBound} easily. For completeness, we now provide a quick proof as follows.

The $n<4$ cases are clearly true. Suppose $n\ge4$. Let $\calO=\calM_n$ be the MC-circle of $n$ with radius $r=R_n$ if $n$ is MC. If $n$ is non-MC let $\calO$ be an $n$-enclosing circle of radius $r<R_n$ where $r$ can be made arbitrarily close to $R_n$ by Theorem \ref{thm:Non-MCLatticeCircle}. Let us shade a lattice square if its lower left corner is an interior point of $\calO$. Then the shaded area is exactly $n$. We now show that we we reduce the radius of $\calO$ by any $d>\sqrt{2}$ while fixing the center of $\calO$ , then all of the circumference will lie in the shaded area.
\begin{figure}[H]
\centering
  \includegraphics[scale=0.6]{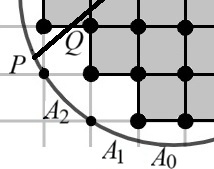}
\caption{Three arc pieces $A_j$, $j=0,1,2$, of the circumference of $\calO$ represent all possible such pieces cut out by lattice squares.}
\label{fig:Gauss}
\end{figure}

Indeed, the circumference of $\calO$ is cut into arc pieces by the lattice squares with at least one corner lattice point contained inside $\calO$. Then each piece not already in the shaded area must have one of the three forms (by applying some rigid motions if necessary) shown in Figure \ref{fig:Gauss}: $A_j$ is bounded by $j$ lattice points on two sides ($j=0,1,2$). Let $P$ be any point on such a piece and $C$ the center of $\calO$. Let $Q$ be the intersection of line $PC$ with the boundary of the shaded area. Then clearly $|PQ|\le \sqrt{2}$ and therefore $\pi (r-\sqrt{2})^2<n$. Taking $r\to R_n$ we get
$R_n\le\sqrt{2}+\sqrt{n/\pi}$. However, the inequality cannot hold since $R_n$ is algebraic while $\sqrt{2}+\sqrt{n/\pi}$ is clearly transcendental.

This completes the proof of the theorem.
\end{proof}

\begin{rem}\label{rem:sloan2}
The sequence $\{\nu(k)\}_{k\ge 1}$ is the sequence A051132 on the OEIS website \cite{Sloane2025}. The difference sequence $\{N(k)-\nu(k)\}_{k\ge 1}$ [A046109] provides the number of lattice points on the circumference of the circle with radius $r$ centered at $0$, where $N(r)$ is defined by \eqref{equ:NumberOfPt}.
\end{rem}

Given any bound $M$, we now may use the following algorithm to determine if $n<M$ is MC or non-MC and then find its MC-radius or LUBOR.
\begin{enumerate}[label=\textbf{Step \arabic*.}]
  \item For each $1\le n\le M$, compute all the radius of $n$-enclosing \textbf{lattice} circles, and then find the large radius, call it $\rho_n$.

  \item Start with the MC number 0. For each $n\ge 1$, suppose we have already determined the cases $<n$.  Suppose $k<n$, $k$ is an MC number and all the numbers between $k$ and $n$ (exclusive) are non-MC. Then we can determine $n$ by Main Theorem \ref{mainthm:AllResults}\ref{mainthm:Case4} (=Corollary \ref{cor:NonMCradiiStrictDecrease}): $n$ is MC if and only if $\rho_n\ge \rho_k$. Moreover, $R_n=\rho_n$ if $n$ is MC and $R_n=\rho_k$ if $n$ is non-MC.
\end{enumerate}

Using the above algorithm, we have found the following MC integers $n\le 100$ with MAPLE:
\begin{align*}
& \{0\text{---}4, 7\text{---}16, 19\text{---}32, 37, 39\text{---}49,
 51, 52, 55\text{---}62, 64\text{---}69, 74\text{---}88, 91\text{---}96, 99, 100\}.
\end{align*}
We also can confirm the rest of the integers $n\le 100$ are all non-MC:
$$\{5,6,17,18, 33,\text{---}36, 38, 50, 53, 54, 63, 70\text{---}73, 89, 90, 97, 98\}.$$
Furthermore, we can even determine $R_n$ exactly for all $n\le 100$.

Essentially the same information can be found for all $n<1100$ in Figure \ref{fig:impactingFactorAll}. They show all the strong numbers as well as all the non-MC numbers in this range.

\begin{figure}[H]
\centering
  \includegraphics[scale=0.48]{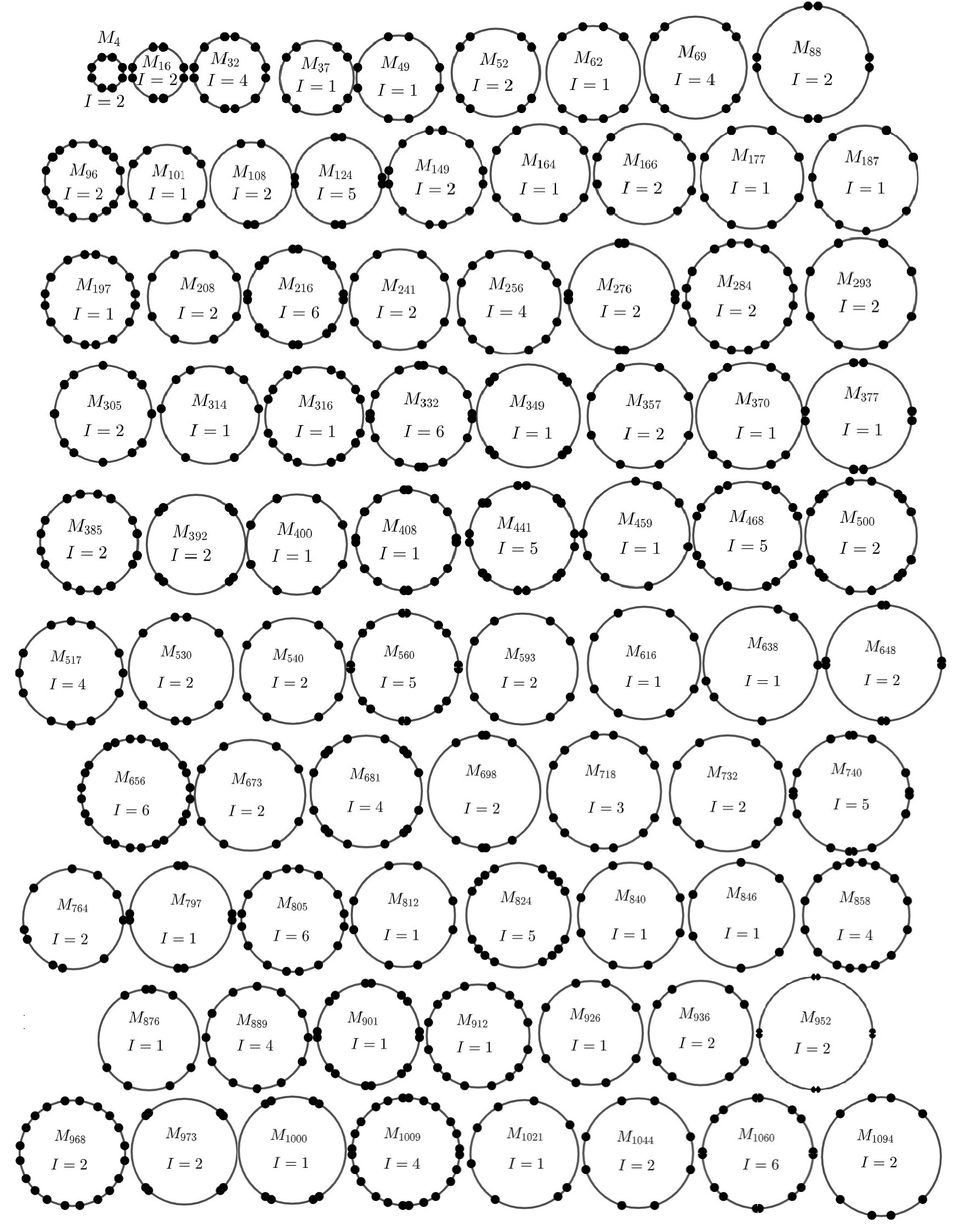}
\caption{All strong MC numbers $n<1100$ with their impacting indices.}
\label{fig:impactingFactorAll}
\end{figure}

From the evidence we found in the range $n<1100$, we make the following conjectures. All the symmetries mentioned below are mirror symmetries keeping the lattice system invariant.
\begin{conj}\label{conj:Main}
We have
\begin{enumerate}[label=(\arabic*)]
  \item \label{cor:MaxCirRadiusStrictIncreasing}    For all MC numbers $k<\ell$ we have $R_k<R_\ell$.
  \item \label{conj:nonMaxLatCircle} For every non-MC $n\ne 6$, there exists a largest $n$-enclosing \textbf{lattice} circle.
  \item Let $n\ne 6$ and let $\rho_n$ be the radius of the largest $n$-enclosing \textbf{lattice} circle if it exists and $\rho_n=0$ otherwise. Then the numbers in $\{\rho_n\}_{n\ge 0}$ are pairwise distinct.
  \item There are arbitrary long consecutive MC integer sequences.
  \item \label{conj:impactingFactorUnbound} There are arbitrary long consecutive non-MC integer sequences, namely, the impacting indices are unbounded.
  \item The density of MC integers among nonnegative integers is greater than 80\%.
  \item The density of non-MC integers among nonnegative integers is positive, but less than 20\%.
  \item There are infinitely many strong MC numbers, with a positive density among all MC numbers.
  \item Every MC-circle of strong MC number has at least six lattice points on its circumference.
  \item For the majority of the MC numbers, the lattice points on the circumference of its MC-circle have at least one mirror symmetry. This density should be over 50\%.
  \item There are infinitely many MC numbers $n$ such that $\calM_n$ has exactly three lattice points on its circumference which have no lattice invariant mirror symmetry. For the vast majority of all such $n$'s (with density $>$98\%), the lattice points form a scalene triangle.
  \item There are very few MC numbers $n$ whose MC-circle has four lattice points without any mirror symmetry, the density should be $<$ 2\%.
  \item The lattice points on the circumference of every MC-circle of strong MC $n$ have at least one mirror symmetry.
  \item Less than 10\% of the MC-circles of strong $n$ among all strong MC-circles have slant $45^\circ$ axes of symmetry.
  \item Around 20\% of the MC-circles among all MC-circles have slant $45^\circ$ axes of symmetry.
  \item Less than 10\% of the MC-circles of strong $n$ among all strong MC-circles have only one horizontal or vertical axis of symmetry.
  \item More than 20\% of the MC-circles among all MC-circles have only one horizontal or vertical axis of symmetry.
  \item More than 80\% of the MC-circles of strong $n$ among all strong MC-circles have at least two axes of symmetry.
  \item Less than 15\% of the MC-circles among all MC-circles have at least two axes of symmetry.
\end{enumerate}
\end{conj}

By computer search, we have the following data to support the above conjecture.
\begin{enumerate}[label=(\arabic*)]
  \item For all MC numbers  $k<\ell<1111$ (the actual range we computed) we have $R_k<R_\ell$.
  \item For every $n\ne 6$, $0\le n<1111$ there is a largest $n$-enclosing lattice circle among all $n$-enclosing lattice circle.
  \item This holds for all $0\le n<1111$. In particular, it implies the proceeding conjecture which is equivalent to saying $\rho_n\ne 0$ unless $n=6$.
  \item The 32 consecutive numbers from 410 to 441, inclusive, are all MC.
  \item This may have the weakest support. The largest impacting inex we found so far is 6.
  \item There are 911 MC integers in the range $0\le n<1100$, about 83\%. Table \ref{tab:distribMax} shows their distributions over the 100-long intervals:
  \begin{table}[H]
\caption{The distribution of MC integers $n<1100$.}
\label{tab:distribMax}
\centering\begin{tabular}{|c|c|c|c|c|c|c|c|c|c|c|c|}
\hline
$\lfloor n/100 \rfloor$  &   0 &  1  & 2  & 3 & 4 & 5 & 6 & 7 & 8 & 9  &10 \\
\hline
$\sharp\{n|$  $n$ is MC$\}$ &   79& 84& 80& 81& 87& 83& 83& 86& 77& 87 & 84 \\
\hline
\end{tabular}
\end{table}

  \item There are 189 non-MC integers in the range  $0\le n<1100$, about 17\%. Table \ref{tab:distribNonMax} shows  their distribution over the 100-long intervals.
\begin{table}[H]
\caption{The distribution of non-MC integers $n<1100$.}
\label{tab:distribNonMax}
\centering\begin{tabular}{|c|c|c|c|c|c|c|c|c|c|c|c|}
\hline
$\lfloor n/100 \rfloor$  &   0 &  1  & 2  & 3 & 4 & 5 & 6 & 7 & 8 & 9  &10 \\
\hline
$\sharp\{n|$ non-MC$\}$ & 21 & 16  & 20 & 19 & 13 & 17& 17 & 14 & 23 & 13 & 16 \\
\hline
\end{tabular}
\end{table}

  \item In the range $0\le n<1100$, there are 80 strong ones among the 911 MC numbers. Hence, strong ones count about 7.3\% among all nonnegative integer, and about 8.8\% among all MC numbers.
     Table \ref{tab:distribStrong} shows their distribution over the 100-long intervals.
\begin{table}[H]
\caption{The distribution of non-MC integers $n<1100$.}
\label{tab:distribStrong}
\centering\begin{tabular}{|c|c|c|c|c|c|c|c|c|c|c|c|}
\hline
$\lfloor n/100 \rfloor$  &   0 &  1  & 2  & 3 & 4 & 5 & 6 & 7 & 8 & 9  &10 \\
\hline
$\sharp\{n|$ strong MC$\}$ & 10 & 9  &7 & 10 & 5 & 6& 7 & 5& 8 & 7& 6 \\
\hline
\end{tabular}
\end{table}

  \item Notice that $\calM_{638}$ and $\calM_{846}$ are the only two MC-circles of strong MC numbers $<1100$ with six lattice points on their circumference while all the others have at least eight.

  \item The symmetry property seems to hold for the majority of MC-circles. There are 529 such $M_n$'s for all $n<1100$. The density of such MC numbers among all MC numbers should be around 59\%.
      Table \ref{tab:distribAllSymmetricCircles} shows  their distribution over the 100-long intervals.

\begin{table}[H]
\caption{The distribution of symmetric $\calM_n$, $n<1100$.}
\label{tab:distribAllSymmetricCircles}
\centering\begin{tabular}{|c|c|c|c|c|c|c|c|c|c|c|c|}
\hline
$\lfloor n/100 \rfloor$  &   0 &  1  & 2  & 3 & 4 & 5 & 6 & 7 & 8 & 9  &10 \\
\hline
$\sharp\{n|$ symmetric $\calM_n\}$ & 86 & 73 & 63 & 71 & 60 & 66 & 65 & 55 & 74 & 52 & 52\\
\hline
\end{tabular}
\end{table}

  \item For $n<1100$, there are 373 MC numbers whose MC-circle has only three lattice points on its circumference which form a scalene triangle, or an isosceles triangle whose axis of mirror symmetry is not lattice invariant (the latter case is extremely rare since we only found one such MC number: 585). This occurs first at $n=18$ whose MC-circle has only three lattice points on its circumference: $(0, 3), (1, -1), (-3, 2)$. The triangle is clearly scalene and therefore no lattice invariant symmetry can exist.  Table \ref{tab:distribAsymmetricCircles} shows   their distribution over the 100-long intervals.  The density of such $n$'s seems to be around 34\% among all nonnegative integers and 41\% among all MC numbers.

\begin{table}[H]
\caption{The distribution of asymmetric $\calM_n$ with 3 boundary lattice points, $n<1100$.}
\label{tab:distribAsymmetricCircles}
\centering\begin{tabular}{|c|c|c|c|c|c|c|c|c|c|c|c|}
\hline
$\lfloor n/100 \rfloor$  &   0 &  1  & 2  & 3 & 4 & 5 & 6 & 7 & 8 & 9  &10 \\
\hline
$\sharp\{n|$ asymmetric $\calM_n\}$ & 13 & 27 & 37 & 28 & 39 & 34 & 35 & 42 & 26 & 46 & 46  \\
\hline
\end{tabular}
\end{table}

  \item There are only 10 MC numbers $n<1100$ such that $\calM_n$ contains four lattice points without any lattice invariant mirror symmetry: 66, 352, 422, 706, 731, 760, 908, 940, 1019, 1024. The density of such $n$'s seems to be around 1\% among all nonnegative integers. We want to point out that the four lattice points sometime form an isosceles trapezoid, i.e., are equipped with mirror symmetries. This holds for the last 6 numbers in the above list. However, the axes of symmetry are neither horizontal, nor vertical, nor slant with 45$^\circ$. Hence, such symmetries are not lattice invariant.

  \item All the symmetries of strong MC-circles are evident from the pictures in Figure \ref{fig:impactingFactorAll}.

  \item There are only 3 strong MC numbers $n<1100$ whose MC-circles have 45$^\circ$ slant symmetry. The density of such MC numbers among all the strong ones is only about 4\%.

  \item Among the 911 MC numbers $<1100$, there are 176 $\calM_n$'s with 45$^\circ$ slant symmetry. Table \ref{tab:distribSlantCircles} shows  their distribution over the 100-long intervals.  The density of such $n$'s seems to be around 16\% among all nonnegative integers and around 19\% among all MC numbers.
\begin{table}[H]
\caption{The distribution of $\calM_n$'s with 45$^\circ$  slant symmetry, $n<1100$.}
\label{tab:distribSlantCircles}
\centering\begin{tabular}{|c|c|c|c|c|c|c|c|c|c|c|c|}
\hline
$\lfloor n/100 \rfloor$  &   0 &  1  & 2  & 3 & 4 & 5 & 6 & 7 & 8 & 9  &10 \\
\hline
$\sharp\{n|$ slant symmetric $\calM_n\}$ & 19 & 20 & 15 & 15 & 16 & 18 & 16 & 13 & 14 & 18 & 12 \\
\hline
\end{tabular}
\end{table}

  \item Among the 80 strong MC numbers $<1100$, only the following have only one horizontal/vertical symmetry: 108, 314, 846, and 876, only about 5\%.

   \item Among the 911 MC numbers $<1100$, there are 229 $\calM_n$'s with only one horizontal/vertical symmetry.
Table \ref{tab:distrib1Circles} shows their distribution over the 100-long intervals.  The density of such $n$'s seems to be around 23\% among all nonnegative integers and around 25\% among all MC numbers.
\begin{table}[H]
\caption{The distribution of $\calM_n$'s with only one horizontal/vertical symmetry, $n<1100$.}
\label{tab:distrib1Circles}
\centering\begin{tabular}{|c|c|c|c|c|c|c|c|c|c|c|c|}
\hline
$\lfloor n/100 \rfloor$  &   0 &  1  & 2  & 3 & 4 & 5 & 6 & 7 & 8 & 9  &10 \\
\hline
$\sharp\left\{n \left|\aligned & \calM_n\ \text{has only one }\\
        & \text{hor./ver. symmetry} \endaligned \right.
\right\}$ &  22 & 21 & 17 & 24 & 22 & 23 & 23 & 21 & 26 & 12 & 17\\
\hline
\end{tabular}
\end{table}

  \item More than 80\% of the MC-circles of strong $n<1100$ have at least two axes of symmetry. The exact density is 71/80 in this range.

  \item Among the 911 MC numbers $<1100$, there are 123 $\calM_n$'s with at least two axes symmetry.
Table \ref{tab:distrib23Circles} shows  their distribution over the 100-long intervals.  The density of such $n$'s seems to be around 11\% among all nonnegative integers and around 14\% among all MC numbers.
\begin{table}[H]
\caption{The distribution of $\calM_n$'s with  at least two axes symmetry, $n<1100$.}
\label{tab:distrib23Circles}
\centering\begin{tabular}{|c|c|c|c|c|c|c|c|c|c|c|c|}
\hline
$\lfloor n/100 \rfloor$  &   0 &  1  & 2  & 3 & 4 & 5 & 6 & 7 & 8 & 9  &10 \\
\hline
$\sharp\left\{n  \left| \aligned & \calM_n\ \text{has at least }\\
        & \text{two symmetries} \endaligned \right.
\right\}$ & 24 & 16 & 11 & 13 & 9 & 8 & 9 & 7 & 11 & 9 & 7 \\
\hline
\end{tabular}
\end{table}
\end{enumerate}

We further remark that the two special classes of circles $\calS_k$ and $\calT_k$ with quite strong symmetries in Section \ref{sec:SpecialClass} seem to contain mostly MC numbers and many are actually strong ones.

However, case $n=89$ shows that an $n$-enclosing lattice circle with strong symmetry does not always guarantee it is an MC-circle of $n$. See Figure \ref{fig:case89}. In fact, 89 is non-MC and therefore even the largest lattice circle on the right of Figure \ref{fig:case89} is not a largest $89$-enclosing circle since such a circle does not exist.

\begin{figure}[H]
\centering
  \includegraphics[scale=0.7]{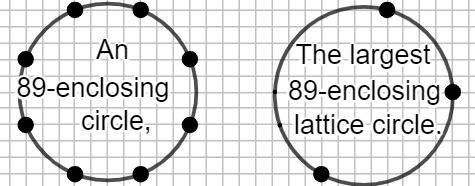}
\caption{A circle with strong symmetry may not be a largest enclosing circle.}
\label{fig:case89}
\end{figure}

\section{Concluding remarks and further research problems}
We have demonstrated in this paper that the naive looking question ``What is the largest circle containing $n$ lattice points in the interior?" has a very delicate answer. For each $n\le 40$ we can solve the problem by hand and for each $n<1100$ we can solve the problem with the aid of a computer. In theory,  via our approach we can determine all the MC and non-MC integers $n$ with a precise computation of the largest radius $\rho_n$ of all the lattice circles enclosing exactly $n$ lattice points, which can be achieved with a finite amount of computation.

The radius sequence $\{\rho_n\}_{n\ge 0}$ is \textbf{not} always increasing. However, whenever a drop after $\rho_n$ occurs there are some associated non-MC numbers. Moreover, the largest strictly increasing subsequence is conjectured to correspond exactly to all the MC numbers while the LUBOR $R_n$ of any non-MC $n$ is proved to be given by the largest radius in the sequence before $\rho_n$. When we only consider the largest \textbf{lattice} $n$-enclosing circles, $n=129$ is the only number less than 1100 such that $\rho_n<\rho_{n-1}$ with both $n$ and $n-1$ are non MC numbers.

There are some obvious patterns and symmetries concerning the MC-circles of $n$, particularly the strong ones, namely, those with trailing non-MC numbers. In Conjecture \ref{conj:Main}, we have listed numerous observations which we believe should hold in general.

There are a few related research problems which deserve some attention.
\begin{probs}\label{prob:concluding}
\begin{enumerate}[label=(\arabic*)]
  \item
We may consider the question of finding the \emph{smallest} $n$-enclosing circles. This problem is closely related to the concept we developed in this paper. For example, we may define similarly the \emph{minimal-enclosing numbers} and \emph{non-minimal-enclosing numbers}.

  \item We may consider the question of finding the \emph{smallest/largest} circles containing $n$ lattice points in the interior \emph{and on the boundary}. The answers will be drastically different.

  \item \label{prob:highDim} We may consider higher dimension analogs of these questions. For example, what is the largest/smallest sphere in 3-dimensional space that encloses exactly $n$ lattice points lattice points in the interior \emph{and/not on the boundary}.

  \item More generally, as  Steinhaus' Theorem \ref{thm:Steinhaus} has been generalized to any Hilbert space $X$ in  \cite{KaniaKochanek2017,Zwolenski2011}, we wonder when the largest $n$-enclosing open balls exist for a quasi-finite set (see  \cite[p.~129]{Zwolenski2011} for the definition). Do we have similar results in this setting?
\end{enumerate}
\end{probs}

Note that Kulikowski  \cite{Kulikowski1959} showed that for every positive integer $n$, there is a three-dimensional sphere which has exactly $n$ lattice points on its surface. This is related to Problem \ref{prob:concluding}\ref{prob:highDim}.

\section*{List of acronyms and key terinology}

\begin{quote}
\emph{MC number}: maximally circlable number, i.e., a nonnegative number $n$ such that there is a largest circle to contain $n$ lattice points in the interior

\emph{Non-MC number}: Non-maximally circlable number, circle as above does not exist

\emph{MC-circle of $n$}: the largest circle to contain $n$ lattice points in the interior

\emph{MC-radius of $n$}: the radius of MC-circle of $n$

\emph{MC-set of $n$}: a set of $n$ lattice points that are enclosed by an MC-circle of $n$

\end{quote}

\end{document}